\documentclass{article}

\usepackage{amsmath}
\usepackage{amsthm}
\usepackage{amssymb}
\usepackage{amsfonts}

\usepackage{subeqnarray}
\usepackage{bm}
\usepackage{dsfont}
\usepackage{upgreek}
\usepackage{booktabs}
\usepackage{fancyhdr}
\usepackage{multirow}
\usepackage{multicol}
\usepackage{float}
\usepackage{enumerate}
\usepackage{cases}

\usepackage[ruled,vlined,linesnumbered,inoutnumbered,algo2e]{algorithm2e}

\usepackage{bbding}
\usepackage{caption}
\usepackage{todonotes}
\usepackage{xspace}

\usepackage{mathrsfs}

\usepackage{lineno}

\usepackage{xurl}

\usepackage[top=2.5cm, bottom=2.5cm, left=3cm, right=3cm]{geometry} 

\usepackage{graphics,graphicx,epsf,epstopdf,subfigure}
\graphicspath{{./Figures/}}
\ifpdf
\DeclareGraphicsExtensions{.eps,.pdf,.png,.jpg}
\else
\DeclareGraphicsExtensions{.eps}
\fi

\usepackage[breaklinks,colorlinks,linkcolor=black,citecolor=black,urlcolor=black,hypertexnames=false]{hyperref}

\providecommand{\U}[1]{\protect\rule{.1in}{.1in}}

\newtheorem{theorem}{Theorem}[section]
\newtheorem{corollary}[theorem]{Corollary}
\newtheorem{lemma}[theorem]{Lemma}
\newtheorem{proposition}[theorem]{Proposition}
\newtheorem{definition}[theorem]{Definition}
\newtheorem{remark}{Remark}
\newtheorem{assumption}{Assumption}

\numberwithin{equation}{section}

\newcommand{\dist}{\mathsf{dist}}

\newcommand{\Diag}{\mathsf{Diag}}

\newcommand{\tr}{\mathsf{tr}}

\newcommand{\dom}{\mathrm{dom}}
\newcommand{\st}{\mathrm{s.\,t.}\,\,}

\newcommand{\grad}{\mathsf{grad}}

\newcommand{\bK}{\mathbb{K}}
\newcommand{\bN}{\mathbb{N}}
\newcommand{\bR}{\mathbb{R}}

\newcommand{\cC}{\mathcal{C}}

\newcommand{\cN}{\mathcal{N}}
\newcommand{\cO}{\mathcal{O}}

\newcommand{\cS}{\mathcal{S}}

\newcommand{\cU}{\mathcal{U}}

\newcommand{\Rn}{\mathbb{R}^{n}}

\newcommand{\Rnp}{\mathbb{R}^{n \times p}}

\newcommand{\Rnn}{\mathbb{R}^{n \times n}}  
\newcommand{\Rnm}{\mathbb{R}^{n \times m}}
\newcommand{\Rmn}{\mathbb{R}^{m \times n}}
\newcommand{\Rpnp}{\mathbb{R}^{n \times p}_{+}}

\newcommand{\Sphn}{\mathcal{O}^{n,1}}
\newcommand{\Sphpn}{\mathcal{O}^{n,1}_{+}}
\newcommand{\Onp}{\mathcal{O}^{n,p}}
\newcommand{\Opp}{\mathcal{O}^{p,p}}
\newcommand{\Opnp}{\mathcal{O}^{n,p}_{+}}
\newcommand{\Opnn}{\mathcal{O}^{n,n}_{+}}
\newcommand{\Qnp}{\mathcal{Q}^{n,p}}
\newcommand{\Qpnp}{\mathcal{Q}^{n,p}_{+}}

\newcommand{\zz}{^{\top}}

\newcommand{\ff}{_{\mathsf{F}}}

\newcommand{\ffs}{^2_{\mathsf{F}}}

\newcommand{\sign}{\mathsf{sign}}
\newcommand{\zrow}{\mathsf{zrow}}
\newcommand{\srow}{\mathsf{srow}}
\newcommand{\supp}{\mathsf{supp}}

\newcommand{\map}{\mathsf{map}}
\newcommand{\alg}{\mathsf{alg}}
\newcommand{\opt}{\mathsf{opt}}
\newcommand{\init}{\mathsf{init}}

\newcommand{\uast}{^{\ast}}
\newcommand{\last}{_{\ast}}

\newcommand{\dkh}[1]{\left(#1\right)}
\newcommand{\hkh}[1]{\left\{#1\right\}}

\newcommand{\jkh}[1]{\left\langle#1\right\rangle}
\newcommand{\norm}[1]{\left\|#1\right\|}
\newcommand{\abs}[1]{\left\lvert #1\right\rvert}

\newcommand{\SUPPORT}{\texttt{Support-Set}\xspace}
\newcommand{\EPOrth}{\texttt{EP4Orth+}\xspace}
\newcommand{\SEPPGz}{\texttt{SEPPG0}\xspace}
\newcommand{\SEPPGp}{\texttt{SEPPG+}\xspace}
\newcommand{\OptM}{\texttt{OptM}\xspace}

\allowdisplaybreaks
\graphicspath{ {./results/} }
\definecolor{Gray}{rgb}{0.5,0.5,0.5}
\DeclareMathOperator*{\argmin}{arg\,min}

\makeatletter

\newcommand{\Rmnum}[1]{\expandafter\@slowromancap\romannumeral #1@}
\makeatother

\newcommand*\patchAmsMathEnvironmentForLineno[1]{
	\expandafter\let\csname old#1\expandafter\endcsname\csname#1\endcsname
	\expandafter\let\csname oldend#1\expandafter\endcsname\csname end#1\endcsname
	\renewenvironment{#1}
	{\linenomath\csname old#1\endcsname}
	{\csname oldend#1\endcsname\endlinenomath}
}
\newcommand*\patchBothAmsMathEnvironmentsForLineno[1]{
	\patchAmsMathEnvironmentForLineno{#1}
	\patchAmsMathEnvironmentForLineno{#1*}
}
\AtBeginDocument{
	\patchBothAmsMathEnvironmentsForLineno{equation}
	\patchBothAmsMathEnvironmentsForLineno{align}
	\patchBothAmsMathEnvironmentsForLineno{flalign}
	\patchBothAmsMathEnvironmentsForLineno{alignat}
	\patchBothAmsMathEnvironmentsForLineno{gather}
	\patchBothAmsMathEnvironmentsForLineno{multline}
	\patchAmsMathEnvironmentForLineno{aligned}
}

\title{A Support-Set Algorithm for Optimization Problems with Nonnegative and Orthogonal Constraints\thanks{This work is supported by National Key R\&D Program of China (2023YFA1009300), RGC grant JLFS/P-501/24 for the CAS AMSS-PolyU Joint Laboratory in Applied Mathematics, CAS-Croucher Funding Scheme for Joint Laboratories, Hong Kong Research Grant Council project PolyU15300024, and National Natural Science Foundation of China (12125108, 12021001, 12288201).}}

\author{
	Lei Wang\thanks{Department of Applied Mathematics, The Hong Kong Polytechnic University, Hung Hom, Kowloon, Hong Kong, China (\href{mailto:wlkings@lsec.cc.ac.cn}{wlkings@lsec.cc.ac.cn}).}
	\and Xin Liu\thanks{State Key Laboratory of Mathematical Sciences, Academy of Mathematics and Systems Science, Chinese Academy of Sciences, and University of Chinese Academy of Sciences, Beijing, China (\href{mailto:liuxin@lsec.cc.ac.cn}{liuxin@lsec.cc.ac.cn}).}
	\and Xiaojun Chen\thanks{Department of Applied Mathematics, The Hong Kong Polytechnic University, Hung Hom, Kowloon, Hong Kong, China (\href{mailto:maxjchen@polyu.edu.hk}{maxjchen@polyu.edu.hk}).}
}

\date{}

\begin{document}

\maketitle

\begin{abstract}
	In this paper, we investigate optimization problems with nonnegative and orthogonal constraints, where any feasible matrix of size $n \times p$ exhibits a sparsity pattern such that each row accommodates at most one nonzero entry. Our analysis demonstrates that, by fixing the support set, the global solution of the minimization subproblem for the proximal linearization of the objective function can be computed in closed form with at most $n$ nonzero entries. Exploiting this structural property offers a powerful avenue for dramatically enhancing computational efficiency. Guided by this insight, we propose a support-set algorithm preserving strictly the feasibility of iterates. A central ingredient is a strategically devised update scheme for support sets that adjusts the placement of nonzero entries. We establish the convergence of the support-set algorithm to a first-order stationary point, and show that its iteration complexity required to reach an $\epsilon$-approximate first-order stationary point is $O (\epsilon^{-2})$. Numerical results are strongly in favor of our algorithm in real-world applications, including nonnegative PCA, clustering, and community detection. 
\end{abstract}

%\noindent\rule{\textwidth}{0.05em}

% ---------------------------------------------------------------------------------------------------------------------------------

\section{Introduction}

Our focus of this paper is on the optimization problems with nonnegative and orthogonal constraints of the following form,
\begin{equation}
	\label{opt:stplus} %\tag{O+}
	\begin{aligned}
		\min_{X \in \Rnp} \hspace{2mm} & f (X) \\
		\st \hspace{3.5mm} & X\zz X = I_p, \; X \geq 0,
	\end{aligned}
\end{equation}
where $f: \Rnp \to \bR$ is the objective function, $I_p$ is the $p \times p$ identity matrix, and the notation $X \geq 0$ represents the entrywise nonnegativity of $X$.
The feasible set of problem \eqref{opt:stplus} is denoted as $\Opnp := \Onp \cap \Rpnp$, where $\Onp := \{X \in \Rnp \mid X\zz X = I_p\}$ is the Stiefel manifold \cite{Absil2008optimization,Boumal2023introduction} in $\Rnp$ and $\Rpnp := \{X \in \Rnp \mid X \geq 0\}$ is the cone of nonnegative matrices in $\Rnp$.
Throughout this paper, we make the following blanket assumption on problem \eqref{opt:stplus}.

\begin{assumption}
	\label{asp:objective}
	The function $f$ is continuously differentiable and its Euclidean gradient $\nabla f$ is Lipschitz continuous over $\Onp$ with the corresponding Lipschitz constant $L \geq 0$.
\end{assumption}

Recently, problems of the form \eqref{opt:stplus} have captured a wide variety of applications and interests in machine learning and data science, such as nonnegative principal component analysis (PCA) \cite{Montanari2015non,Zass2006nonnegative}, nonnegative Laplacian embedding \cite{Luo2009non,Zhang2019nonnegative}, spectral clustering \cite{Boutsidis2009unsupervised,Carson2017manifold,Yang2012discriminative}, and orthogonal nonnegative matrix factorization (ONMF) \cite{Ding2006orthogonal,Kuang2012symmetric,Yang2010linear}.
In particular, problem \eqref{opt:stplus} covers some classical NP-hard problems as special cases, including the problem of checking copositivity of a symmetric matrix \cite{Hiriart2010variational,Povh2007copositive} and the quadratic assignment problem \cite{Lawler1963quadratic,Xia2006new}.

In this paper, we devote our attention to the regime where $1< p < n$.
When $p = 1$, the Stiefel manifold reduces to the unit sphere $\Sphn := \{x \in \Rn \mid x\zz x = 1\}$.
For this special case, the linear independence constraint qualification (LICQ) is satisfied (see \cite[Definition 12.4]{Nocedal2006numerical}).
And we can solve this problem by resorting to the classical projected gradient method, under which the projection onto $\Sphpn$ enjoys an explicit analytical form \cite{Zhang2018sparse}.
In comparison, for $p > 1$, these properties generally fail to hold.
When $p = n$, the feasible set $\Opnn$ coincides with the collection of all $n \times n$ permutation matrices.
%with its cardinality being $n!$.
Consequently, problem \eqref{opt:stplus} with $p = n$ essentially takes on a discrete nature, requiring specialized strategies to attain high-quality solutions.
Furthermore, such problems can be treated by penalty-based approaches \cite{Jiang2016lp}.
For further insights into this setting, we refer interested readers to \cite{Chen2025tight,Jiang2016lp,Qian2024error}.

The nonnegativity in $\Opnp$ destroys the smoothness of $\Onp$ and introduces some combinatorial features.
For a matrix $X \in \Opnp$, \textit{each row has at most one nonzero entry}, and hence, the total number of nonzero entries is at most $n$.
This property arises from the structure that any two nonzero entries within the same row---both strictly positive---would unavoidably disrupt the orthogonality across columns.
We define the \textit{support set} of a matrix as the collection of positions corresponding to its nonzero entries.
Viewed through this lens, the feasible set $\Opnp$ can be partitioned into a finite union of subsets, each distinguished by a unique pattern of the support set.
Once the support set is fixed, problem \eqref{opt:stplus} can be reformulated and addressed in a lower-dimensional space, as its dimensionality is effectively reduced from $np$ to $n$.
Then it is foreseeable that such a reformulation can lead to a substantial enhancement in computational efficiency.
%However, this delicate trait has long remained hidden in the shadows of algorithmic design.
As we shall see in the next subsection, the vast majority of prevailing algorithms leverage infeasible strategies to solve problem \eqref{opt:stplus}, without explicitly accounting for the sparsity structure of $\Opnp$.

Motivated by these observations, we aim to develop a theoretically sound and practically viable algorithm capable of navigating among different support sets.
Unfortunately, this task entails a formidable challenge, as the total number of possible support sets grows exponentially with $n$.
Beyond this, no practical mechanism is available to impose both nonnegativity and orthogonality at the same time, which constitutes a profound obstacle to the effective update of the support set.
In particular, the projection onto $\Opnp$ admits no closed-form expression, and currently one can only resort to infeasible approaches to tackle the associated optimization model.
To break through these impasses, we will take full advantage of the structural property inherent in nonnegative and orthogonal matrices.

\subsection{Prior and Related Works}

\label{subsec:prior}

Although optimization problems over the Stiefel manifold have been extensively explored \cite{Absil2008optimization,Wang2025decentralized, Wang2021multipliers,Wang2025distributionally,Wen2013feasible}, the investigation of algorithms for problem \eqref{opt:stplus} is still restricted that exhibit provable convergence guarantees.
%remains in its infancy
Broadly speaking, existing algorithms can be categorized into two classes.
The first class is tailored to specific instances of problem \eqref{opt:stplus} and can hardly be adapted to the generic setting, including multiplicative update schemes \cite{Ding2006orthogonal,Yang2010linear,Yoo2010orthogonal}, orthogonal pivoting algorithms \cite{Zhang2019greedy}, primal-dual frameworks \cite{Li2026unsupervised,Pompili2014two}, penalty approaches \cite{Jiang2016lp,Li2014two,Wang2021clustering}, and convex relaxation methods \cite{Pan2018orthogonal}.
A detailed exposition of these works falls beyond the scope of this paper.
The second class, by contrast, addresses the general formulation of problem \eqref{opt:stplus}, which we will elaborate on in this subsection.

Throughout this paper, we adopt the notations $[X]_{i, j}$, $[X]_{i, :}$, and $[X]_{:, j}$ to represent the $(i, j)$-th entry, the $i$-th row, and the $j$-th column of a matrix $X$, respectively.
Let $\Qpnp := \Qnp \cap \Rpnp$ with $\Qnp := \{X \in \Rnp \mid [X]_{:, j}\zz [X]_{:, j} = 1 \mbox{ for all } j\}$ being the oblique manifold \cite{Absil2008optimization,Boumal2023introduction} in $\Rnp$.
Jiang et al. \cite{Jiang2023exact} propose an exact penalty approach \EPOrth based on the equivalent description 
\begin{equation}
	\label{eq:oblique}
	\Opnp = \hkh{ X \in \Qpnp \mid \norm{X V}\ff = 1 },
\end{equation}
where $V \in \bR^{p \times r}$, with $r$ being an arbitrary positive integer, is a constant matrix satisfying $\norm{V}\ff = 1$ and $[V V\zz]_{i, j} > 0$ for any $i$ and $j$.
By penalizing the constraint $\norm{XV}\ff = 1$, \EPOrth attempts to solve a series of penalty subproblems of the following form,
\begin{equation}
	\label{opt:ep4orth}
	\min_{X \in \Qpnp} \hspace{2mm} f (X) + \rho \norm{X V}\ffs,
\end{equation}
where $\rho > 0$ is a penalty parameter.
The convergence behavior of \EPOrth critically hinges on the quality of approximate solutions to subproblem \eqref{opt:ep4orth} obtained at each iteration.
It is claimed that \EPOrth converges to a weakly first-order stationary point of problem \eqref{opt:stplus} if subproblem \eqref{opt:ep4orth} is solved to first-order stationarity.
This convergence result can be strengthened to a first-order stationary point of problem \eqref{opt:stplus} provided that each iterate satisfies a weakly second-order stationarity condition of subproblem \eqref{opt:ep4orth}.

Two penalty-based methods, \SEPPGz and \SEPPGp, are introduced by Qian et al. based on the local error bound derived in \cite{Qian2024error}.
Let $\rho > 0$ remain a penalty parameter.
Each iteration of \SEPPGz and \SEPPGp solves the following penalty subproblems,
\begin{equation}
	\label{opt:seppg0}
	\min_{X \in \Onp} \hspace{2mm} f (X) + \rho \norm{ \max \hkh{0, - X} }\ffs,
\end{equation}
and
\begin{equation}
	\label{opt:seppg+}
	\min_{X \in \Onp, Y \in \Rnp} \hspace{2mm} f (X) + \rho \norm{ \max \{0, - Y\} }_1 + \dfrac{\rho}{2 \gamma} \norm{ Y - X }\ffs,
\end{equation}
respectively.
Here, $\gamma > 0$ is a constant.
It is proved in \cite{Qian2024error} that problems \eqref{opt:seppg0} and \eqref{opt:seppg+} serve as global exact penalty models of problem \eqref{opt:stplus} when $f$ fulfills a lower second-order calmness condition and every global minimizer contains no zero rows.
The convergence of \SEPPGz and \SEPPGp to a first-order stationary point of problem \eqref{opt:stplus} is demonstrated, if each penalty subproblem is solved to first-order stationarity.

Apart from the aforementioned algorithms, there are studies that further lend support to the construction of penalty models for problem \eqref{opt:stplus}.
In a recent work, Chen et al. \cite{Chen2025tight} establish a global and tight error bound for a class of sign-constrained Stiefel manifolds, which includes $\Opnp$ as a special case.
Their error bound features an exponent of $1/2$ and cannot be improved.
This result explains why square-root terms are necessary in the error bound for problem \eqref{opt:stplus} with $1 < p < n$.
It also remains a challenging endeavor to solve the associated penalty model.
Moreover, drawing on an alternative characterization $\Opnp = \{X \in \Qpnp \mid \norm{[X]_{i, :}}_r = \norm{[X]_{i, :}}_s \mbox{ for all } i\}$ with $1 \leq r < s$, Wang et al. \cite{Wang2021clustering} develop a nonconvex penalty method tailored to the ONMF problem.
While their theoretical analysis rests on the special structure of the objective function, the underlying methodology of formulating penalty models naturally lends itself to broader generalizations.

Finally, it is worth remarking that some recent works consider constrained optimization problems on manifolds.
For instance, Riemannian variants of the augmented Lagrangian method are investigated in \cite{Liu2020simple,Wen2013feasible,Zhou2023semismooth}, whose convergence, however, depends on certain constraint qualifications that are not satisfied by problem \eqref{opt:stplus}.
In the setting where the feasible set can be expressed as the intersection of a smooth manifold and a convex set, the studies by Ding and Toh \cite{Ding2025exploration} and by Xiao et al. \cite{Xiao2025exact} explore interior-point and penalty-based approaches, respectively.
Nevertheless, these frameworks also prove inapplicable to problem \eqref{opt:stplus}, as its feasible set $\Opnp$ lacks any interior points and fails to meet the required nondegeneracy conditions of constraints.

\subsection{Contribution}

In this paper, we capitalize on the structural feature inherent in $\Opnp$ to devise a conceptually new \textit{support-set algorithm} for problem \eqref{opt:stplus}.
The proposed approach ensures that the sequence of iterates remains feasible.
The fundamental philosophy behind our algorithm lies in pursuing a support set that surpasses the current one, which bears a certain resemblance to the classical active-set methods \cite{Nocedal2006numerical} in spirit.
For the present setting, the active set of a matrix comprises the positions of zero entries.
The support set of a matrix can be viewed as the complement of its active set.
However, conventional techniques for updating the active set cannot preserve the special structure of nonnegative and orthogonal matrices.
In addition, the cardinality of a support set (at most $n$) is much smaller than that of an active set (at least $np - n$).

At its core, the support-set algorithm proceeds by minimizing a proximal linearization of the objective function within a fixed support set, and invokes a novel update scheme for support sets whenever sufficient reduction in the objective function value is not achieved.
In particular, the nonzero entries that are close to zero are reallocated to alternative columns that yield a further reduction in the objective function value.
Meanwhile, for rows of all zeros, we design a refined strategy to judiciously activate a position for nonzero entries.
These mechanisms effectively exploit the property that each row accommodates at most one nonzero entry, thus safeguarding feasibility at every iteration.
The update scheme for support sets not only adjusts the placement of nonzero entries but also drives a more pronounced reduction in the objective function value.
In sharp contrast to existing methods, each iteration of the support-set algorithm engages with only $n$ nonzero entries.
Remarkably, all subproblems admit explicit closed-form solutions, whose evaluation, owing to the underlying sparsity structure of matrices, demands only negligible computational effort.

To the best of our knowledge, the support-set algorithm is the first feasible approach for problem~\eqref{opt:stplus} that comes with provable theoretical guarantees.
We rigorously establish its subsequence convergence to a first-order stationary point. 
Our algorithm possesses both a sufficient descent property and the ability to control the stationarity violation by the distance between consecutive iterates. 
These two ingredients enable us to demonstrate the convergence of the whole sequence generated by our algorithm when the objective function $f$ is semi-algebraic. 
This analysis relies essentially on the semi-algebraic nature of the feasible set $\Opnp$.
Furthermore, the iteration complexity required to reach an $\epsilon$-approximate first-order stationary point is $O (\epsilon^{-2})$. 
As far as is known, the results concerning iteration complexity for problem \eqref{opt:stplus} have not yet appeared in the existing literature.
Finally, extensive numerical experiments substantiate the superior performance of our algorithm across a variety of benchmark tasks.
By fully leveraging the sparsity structure, the support-set algorithm achieves striking computational gains, realizing up to an order-of-magnitude speedup over state-of-the-art approaches in practical applications.

\subsection{Organization}

The rest of this paper proceeds as follows.
Section \ref{sec:preliminary} draws into basic notations and stationarity conditions.
We show in Section \ref{sec:support} that minimizing the proximal linearization of the objective function on a fixed support set yields a closed-form solution.
Section \ref{sec:algorithm} is dedicated to devising a novel support-set algorithm to solve problem \eqref{opt:stplus}.
And we establish the subsequence convergence and iteration complexity of the proposed algorithm in Section \ref{sec:convergence}.
Numerical results are presented in Section \ref{sec:numerical} to corroborate the superior computational efficiency of our algorithm relative to existing methods.
Finally, we give concluding remarks in Section \ref{sec:conclusion}.

\section{Preliminaries}

\label{sec:preliminary}

In this section, we introduce the notations used throughout this paper, delineate the stationarity conditions of problem \eqref{opt:stplus}, and present the concepts of Kurdyka-{\L}ojasiewicz property.

\subsection{Basic Notations}

We use $\bR$ and $\bN$ to denote the sets of real and natural numbers, respectively.
%And the notations $\bR_{+}$ and $\bR_{++}$ represent the sets of nonnegative and positive real numbers, respectively.
The Euclidean inner product of two matrices $A_1$ and $A_2$ with the same size is defined as $\jkh{A_1, A_2} = \tr(A_1\zz A_2)$, where $\tr (B)$ stands for the trace of a square matrix $B$.
We denote by $\Diag (B)$ the diagonal matrix whose diagonal entries coincide with those of a square matrix $B$.
%We adopt $I_p \in \Rpp$ to represent the $p \times p$ identity matrix.
The Frobenius norm and the $\ell_q$ norm with $q \geq 1$ of a matrix $C$ are represented by $\norm{C}\ff$ and $\norm{C}_q$, respectively. 
We define the distance between a point $X$ and a set $\cS$ by $\dist (X, \cS) := \inf \{\norm{Y - X}\ff \mid Y \in \cS\}$. 
%We adopt $[C]_{i, j}$, $[C]_{i, :}$, and $[C]_{:, j}$ to denote the $(i, j)$-th entry, the $i$-th row, and the $j$-th column of a matrix $C$, respectively.
The notation $\supp (C) := \{(i, j) \mid [C]_{i, j} \neq 0\}$ refers to the support set of a matrix $C$ and $\zrow (C) := \{i \mid \norm{[C]_{i, :}}_2 = 0\}$ represents the collection of indices corresponding to the zero rows of a matrix $C$.
The sign matrix $\sign (C)$ is defined entrywise by $[\sign (C)]_{i, j} = 1$ if $[C]_{i, j} > 0$, $[\sign (C)]_{i, j} = -1$ if $[C]_{i, j} < 0$, and $[\sign (C)]_{i, j} = 0$ if $[C]_{i, j} = 0$.
We denote by $\odot$ the Hadamard product.
%Given a differentiable function $f: \Rnp \to \bR$, the Euclidean gradient of $f$ with respect to $X$ is denoted by $\nabla f (X)$.
Further notations will be introduced wherever they occur.

\subsection{Stationarity Condition}

The feasible set $\Opnp$ of problem~\eqref{opt:stplus} exhibits an inherent interplay between discrete and continuous structures when $1 < p < n$. 
In line with most existing studies, we investigate problem~\eqref{opt:stplus} from the perspective of continuous optimization. 
A point $X\last \in \Opnp$ is called a local minimizer of problem~\eqref{opt:stplus} if there exists a constant $\nu > 0$ such that $f (X\last) \leq f (X)$ for all $X \in \Opnp$ with $\norm{X - X\last}\ff \leq \nu$. 
This subsection is devoted to deriving necessary conditions for a point to be a local minimizer.

As discussed in \cite[Section 2.3]{Jiang2023exact}, the Guignard constraint qualification \cite{Andreani2016cone} always holds for problem~\eqref{opt:stplus}. 
Every local minimizer $X\last \in \Opnp$ of problem~\eqref{opt:stplus} necessarily satisfies the following condition,
\begin{equation}
	\label{eq:kkt}
	0 \in \nabla f (X\last) + \cN_{\Opnp} (X\last),
\end{equation}
where $\cN_{\Opnp} (X\last)$ is the normal cone to $\Opnp$ at $X\last$. 
The explicit expression of $\cN_{\Opnp} (X)$ is derived in \cite[Lemma~2.2]{Jiang2023exact} as follows,
\begin{equation}
	\label{eq:normal}
	\cN_{\Opnp} (X)
	= \hkh{
		D \in \Rnp \;\middle|\;
		\begin{aligned}
			& [D]_{i, j} = \lambda_j [X]_{i, j} \mbox{ with } \lambda_j \in \bR, \mbox{ for all } (i, j) \in \supp (X), \\
			& [D]_{i, j} \leq 0, \mbox{ for all } i \in \zrow(X) \mbox{ and } j \in \{1, 2, \dotsc, p\}
		\end{aligned}
	}.
\end{equation}
Below we provide an equivalent characterization for \eqref{eq:kkt}.

\begin{proposition}
	\label{prop:kkt}
	A point $X\last \in \Opnp$ satisfies the condition \eqref{eq:kkt} if and only if the following relationships hold,
	\begin{subnumcases}{\label{eq:sc}}
		\label{eq:sc-supp}
		{[\grad\, f (X\last)]_{i\last, j\last} = 0}, 
		& for all $(i\last, j\last) \in \supp (X\last)$, \\
		\label{eq:sc-zrow}
		{[\nabla f (X\last)]_{i\last, j\last} \geq 0}, 
		& for all $i\last \in \zrow (X\last)$ and $j\last \in \{1, 2, \dotsc, p\}$,
	\end{subnumcases} %subnumcases
	where $\grad\, f (X\last) = \nabla f (X\last) - X\last \Diag (X\last\zz \nabla f (X\last))$ represents the Riemannian gradient of $f$ at $X\last$. 
\end{proposition}

\begin{proof}
	We begin by assuming that $X\last$ satisfies the condition \eqref{eq:kkt}, which indicates that $- \nabla f (X\last) \in \cN_{\Opnp} (X\last)$.
	Then, for any $(i\last, j\last) \in \supp (X\last)$, there exists $\lambda_{j\last} \in \bR$ such that 
	\begin{equation*}
		[- \nabla f (X\last)]_{i\last, j\last} = \lambda_{j\last} [X\last]_{i\last, j\last}.
	\end{equation*}
	Straightforward calculations yield that
	\begin{equation*}
		\begin{aligned}
			{[X\last\zz \nabla f (X\last)]_{j\last, j\last}}
			= {} & \sum_{i = 1}^{n} [X\last]_{i, j\last} [\nabla f (X\last)]_{i, j\last}
			= \sum_{i: (i, j\last) \in \supp ([X\last])} [X\last]_{i, j\last} [\nabla f (X\last)]_{i, j\last} \\
			= {} & - \lambda_{j\last} \sum_{i: (i, j\last) \in \supp ([X\last])} [X\last]_{i, j\last}^2 
			= - \lambda_{j\last}.
		\end{aligned}
	\end{equation*}
	We can obtain that
	\begin{equation*}
		\begin{aligned}
			{[\grad\, f (X\last)]_{i\last, j\last}}
			= {} & [\nabla f (X\last)]_{i\last, j\last} - [X\last]_{i\last, j\last} [X\last\zz \nabla f (X\last)]_{j\last, j\last} \\
			= {} & [\nabla f (X\last)]_{i\last, j\last} + \lambda_{j\last} [X\last]_{i\last, j\last}
			= 0,
		\end{aligned}
	\end{equation*}
	for all $(i\last, j\last) \in \supp (X\last)$. 
	Hence, the condition \eqref{eq:sc-supp} holds. 
	Now we choose $i\last \in \zrow (X\last)$. 
	Then the inclusion $- \nabla f (X\last) \in \cN_{\Opnp} (X\last)$ guarantees that $[- \nabla f (X\last)]_{i\last, :} \leq 0$, which is equivalent to the condition \eqref{eq:sc-zrow}. 
	
	Next, we consider that $X\last$ satisfies two relationships in \eqref{eq:sc}. 
	Let $\lambda_j = - [X\last\zz \nabla f (X\last)]_{j, j}$ for all $j \in \{1, 2, \dotsc, p\}$. 
	For $(i\last, j\last) \in \supp (X\last)$, it follows from the condition \eqref{eq:sc-supp} that
	\begin{equation*}
		\begin{aligned}
			0 = {} & {[\grad\, f (X\last)]_{i\last, j\last}}
			= [\nabla f (X\last)]_{i\last, j\last} - [X\last]_{i\last, j\last} [X\last\zz \nabla f (X\last)]_{j\last, j\last} \\
			= {} & [\nabla f (X\last)]_{i\last, j\last} + \lambda_{j\last} [X\last]_{i\last, j\last},
		\end{aligned}
	\end{equation*}
	which implies that $[- \nabla f (X\last)]_{i\last, j\last} = \lambda_{j\last} [X\last]_{i\last, j\last}$. 
	Moreover, the condition \eqref{eq:sc-zrow} indicates that $[- \nabla f (X\last)]_{i\last, :} \leq 0$ for $i\last \in \zrow (X\last)$. 
	We can conclude that $- \nabla f (X\last) \in \cN_{\Opnp} (X\last)$ and $X\last \in \Opnp$ satisfies the condition \eqref{eq:kkt}. 
	The proof is completed. 
\end{proof}

It is noteworthy that Jiang et al. \cite{Jiang2023exact} also establish the stationarity conditions in \eqref{eq:sc} by exploiting the KKT system related to the alternative description of $\Opnp$ in \eqref{eq:oblique}. 
By contrast, the proof of Proposition~\ref{prop:kkt} proceeds directly from the normal cone of $\Opnp$. 
The notions of stationary points given below follows those adopted in \cite{Jiang2023exact}.

\begin{definition}
	A point $X\last \in \Opnp$ is called a weakly first-order stationary point of problem~\eqref{opt:stplus} if it adheres to the condition \eqref{eq:sc-supp}. 
	Moreover, we say a point $X\last \in \Opnp$ is a first-order stationary point of problem~\eqref{opt:stplus} if it satisfies two conditions \eqref{eq:sc-supp} and \eqref{eq:sc-zrow}. 
\end{definition}

The two conditions \eqref{eq:sc-supp} and \eqref{eq:sc-zrow} play distinct yet complementary roles. 
The former rules out descent directions confined to the current support pattern, whereas the latter excludes descent directions induced by activating entries in zero rows. 
Their combination therefore guarantees stationarity with respect to the entire feasible set.
Building upon the preceding definition, we finally introduce the concept of an approximate first-order stationary point.

\begin{definition}
	A point $X\last \in \Opnp$ is called an $\epsilon$-approximate first-order stationary point of problem \eqref{opt:stplus} if the following conditions hold,
	\newcounter{saveeq}
	\setcounter{saveeq}{\value{equation}}
	\begin{subnumcases}{}
		{\abs{[\grad\, f (X\last)]_{i\last, j\last}} \leq \epsilon},
		& for all $(i\last, j\last) \in \supp (X\last)$, \nonumber \\
		{[\nabla f (X\last)]_{i\last, j\last} \geq - \epsilon},
		& for all $i\last \in \zrow (X\last)$ and $j\last \in \{1, 2, \dotsc, p\}$. \nonumber
	\end{subnumcases}
	\setcounter{equation}{\value{saveeq}}
\end{definition}

The above definition arises from a perturbation of the stationarity conditions in \eqref{eq:sc}. 
Obviously, when $\epsilon = 0$, an $\epsilon$-approximate first-order stationary point precisely satisfies these conditions, thereby qualifying as a first-order stationary point of problem~\eqref{opt:stplus}.

\subsection{Kurdyka-{\L}ojasiewicz Property}

A part of the convergence results developed in this paper falls in the scope of a general class of functions that satisfy the Kurdyka-{\L}ojasiewicz (K{\L}) property \cite{Bolte2007lojasiewicz,Attouch2010proximal}. 
Below, we introduce the basic elements to be used in the subsequent analysis.

For any $\mu > 0$, we denote by $\Phi_\mu$ the class of all concave and continuous functions $\chi: [0, \mu) \to \bR_+$ which satisfy the following conditions.
\begin{enumerate}[(i)]
	
	\item $\chi (0) = 0$;
	
	\item $\chi$ is continuously differentiable on $(0, \mu)$ and continuous at $0$;
	
	\item $\chi^{\prime} (t) > 0$ for any $t \in (0, \mu)$.
	
\end{enumerate}
Now we define the K{\L} property.

\begin{definition}[{\cite[Definition~3.1]{Attouch2010proximal}}]
	Let $\varphi: \Rnp \to (-\infty, +\infty]$ be a proper and lower semi-continuous function and $\partial \varphi$ be the limiting subdifferential of $\varphi$. 
	\begin{enumerate}[(i)]
		
		\item The function $\varphi$ is said to satisfy the K{\L} property at $X \in \dom (\partial \varphi) := \{Y \in \Rnp \mid \partial \varphi (Y) \neq \emptyset\}$ if there exists a constant $\mu \in (0, +\infty]$, a neighborhood $\cU$ of $X$, and a function $\chi \in \Phi_{\mu}$, such that for any $Y \in \cU$ satisfying $\varphi (X) < \varphi (Y) < \varphi (X) + \mu$, the following K{\L} inequality holds,
		\begin{equation*}
			\chi^{\prime} (\varphi (Y) - \varphi (X)) \, \dist (0, \partial \varphi (Y)) \geq 1.
		\end{equation*}
		The function $\chi$ is called a desingularizing function of $\varphi$ at $X$. 
		
		\item We say $\varphi$ is a K{\L} function if $\varphi$ satisfies the K{\L} property at each point of $\dom (\partial \varphi)$. 
		
	\end{enumerate}
\end{definition}

K{\L} functions cover a wealth of nonconvex nonsmooth functions and are ubiquitous in many practical applications. 
%For example, tame functions constitutes a wide class of K{\L} functions, including semi-algebraic and real sub-analytic functions. 
We present a class of semi-algebraic functions that enjoy the K{\L} property.

\begin{definition}[{\cite[Definition~5]{Bolte2014proximal}}]
	A set $\cS \subseteq \Rnp$ is called semi-algebraic if there exists a finite number of real polynomial functions $g_{i, j}: \Rnp \to \bR$ and $h_{i, j}: \Rnp \to \bR$ such that
	\begin{equation*}
		\cS = \bigcup_{j = 1}^{v} \bigcap_{i = 1}^{u} \hkh{X \in \Rnp \mid g_{i, j} (X) = 0 \mbox{ and } h_{i, j} (X) < 0}.
	\end{equation*}
	Moreover, we say a function $\varphi: \Rnp \to (- \infty, + \infty]$ is semi-algebraic if its graph
	\begin{equation*}
		\hkh{(X, t) \in \Rnp \times \bR \mid \varphi (X) = t}
	\end{equation*}
	is semi-algebraic in $\Rnp \times \bR$. 
\end{definition}

Real polynomial functions and indicator functions of semi-algebraic sets are semi-algebraic. 
In addition, finite sums, products, and compositions of semi-algebraic functions are also semi-algebraic. 
We refer interested readers to the references \cite{Bolte2007lojasiewicz,Attouch2009convergence,Attouch2010proximal,Bolte2014proximal} for more details.

\section{Problem \eqref{opt:stplus} With a Fixed Support Set}

\label{sec:support}

As mentioned earlier, any matrix $X \in \Opnp$ has at most one nonzero entry in each row.
Therefore, once the support set of matrices is predetermined, the orthogonality among columns is preserved.
This structure allows us to shift our focus entirely to enforcing the nonnegativity and unit-norm constraints on the individual column vectors.
Leveraging this structure of $\Opnp$, we investigate in this section how to effectively reduce the objective function value of problem \eqref{opt:stplus} with a fixed support set.
This goal is achieved by solving a subproblem that minimizes a proximal linearization of the objective function.
Owing to the structural simplicity induced by the support set, the resulting subproblem admits a closed-form solution.
%which significantly enhances computational efficiency.

We adopt the notation $\bar{f}_{Z}: \Rnp \to \bR$ to represent the proximal linearization of the objective function $f$ around a point $Z \in \Opnp$ as follows,
\begin{equation*}
	\bar{f}_{Z} (X) := f (Z) + \jkh{\nabla f (Z), X - Z} + \dfrac{\eta}{2} \norm{X - Z}\ffs,
\end{equation*}
where $\eta > L$ is a proximal parameter.
The lemma below demonstrates that a sufficient reduction in the function value can be realized through minimizing the proximal linearization.

\begin{lemma}
	\label{le:lip-f}
	Suppose that $Z \in \Opnp$ and $X \in \Opnp$ satisfy $\bar{f}_{Z} (X) \leq \bar{f}_{Z} (Z)$.
	Then we have
	\begin{equation*}
		f (Z) - f (X)
		\geq \dfrac{\eta - L}{2} \norm{X - Z}\ffs.
	\end{equation*}
\end{lemma}

\begin{proof}
	According to the Lipschitz continuity of $\nabla f$, it follows that
	\begin{equation}
		\label{eq:lip-f}
		f (X)
		\leq f (Z)
		+ \jkh{\nabla f (Z), X - Z}
		+ \dfrac{L}{2} \norm{X - Z}\ffs.
	\end{equation}
	As a direct consequence of the relationship $\bar{f}_{Z} (X) \leq \bar{f}_{Z} (Z)$, we can proceed to show that
	\begin{equation}
		\label{eq:des-barf}
		\jkh{\nabla f (Z), X - Z}
		\leq - \dfrac{\eta}{2} \norm{X - Z}\ffs.
	\end{equation}
	Collecting two inequalities \eqref{eq:lip-f} and \eqref{eq:des-barf} together yields the assertion of this lemma.
	We complete the proof.
\end{proof}

It is important to emphasize that, minimizing the proximal linearization of $f$ across the entire feasible set $\Opnp$ is an intractable task.
%as this subproblem lacks a closed-form solution and, at present, can only be tackled by certain infeasible methods.
Nevertheless, we observe that an explicit solution readily emerges by confining the corresponding subproblem to a predetermined support set.

Let $S \in \sign (\Opnp) := \{\sign (X) \mid X \in \Opnp\}$ be a sign matrix of an element in $\Opnp$.
This implies that each row of $S$ contains at most one entry equal to $1$ with all other entries being $0$, and that each column contains at least one entry equal to $1$.
The feasible set $\Opnp$ of problem \eqref{opt:stplus} is characterized by three types of constraints, namely,
\begin{subnumcases}{}
	\label{con:oblique}
	X \geq 0, \quad [X]_{:, j}\zz [X]_{:, j} = 1 \mbox{ for all } j, \\
	\label{con:orth}
	[X]_{:, j}\zz [X]_{:, l} = 0 \mbox{ for all } j \neq l.
\end{subnumcases}
By further imposing a support constraint $\supp (X) \subseteq \supp (S)$, the orthogonality across columns prescribed in \eqref{con:orth} is automatically guaranteed.
As a result, only the two constraints in \eqref{con:oblique} remain to be addressed, whose combination is far more tractable.
And the original problem is essentially reduced to an optimization model on the oblique manifold.
This insight, in turn, naturally leads us to the following problem,
\begin{equation}
	\label{opt:subp-supp}
	\begin{aligned}
		\min_{X \in \Opnp} \hspace{2mm} & \bar{f}_{Z} (X) \\
		\st \hspace{3mm} & \supp (X) \subseteq \supp (S).
	\end{aligned}
\end{equation}
The above formulation simplifies the original problem~\eqref{opt:stplus} in two respects, including the structure of the objective function and the restriction on the support set.

The proposition below unveils that the global minimizer of problem~\eqref{opt:subp-supp} can be computed in closed form, which serves as a cornerstone in the design of our algorithm.

\begin{proposition}
	\label{prop:supp}
	Let $W = \max \{0, ( \eta Z - \nabla f (Z) ) \odot S\} \in \Rpnp$.
	For all $j \in \{1, 2, \dotsc, p\}$, we denote
	\begin{equation*}
		\alpha_j =
		\left\{
		\begin{aligned}
			& - \norm{[W]_{:, j}}_2,
			&& \mbox{if } [W]_{:, j} \neq 0, \\
			& [\nabla f (Z) - \eta Z]_{\bar{i}^{(j)}, j},
			&& \mbox{otherwise},
		\end{aligned}
		\right.
	\end{equation*}
	where $\bar{i}^{(j)} = \min \{i\uast \mid i\uast \in \argmin_{i \in \supp ([S]_{:, j})} [\nabla f (Z) - \eta Z]_{i, j} \}$.
	Then, for any $Z \in \Opnp$ and $S \in \sign (\Opnp)$, the global minimum of problem \eqref{opt:subp-supp} is
	\begin{equation*}
		\bar{f}_Z\uast = f (Z) - \jkh{\nabla f (Z), Z} + \eta p + \sum_{j = 1}^{p} \alpha_j.
	\end{equation*}
	Moreover, it is attained at $\bar{X} \in \Opnp$ whose $j$-th column, for all $j \in \{1, 2, \dotsc, p\}$, takes the form of
	\begin{equation}
		\label{eq:sol-supp}
		[\bar{X}]_{:, j} =
		\left\{
		\begin{aligned}
			& \dfrac{[W]_{:, j}}{\norm{[W]_{:, j}}_2},
			&& \mbox{if } [W]_{:, j} \neq 0, \\
			& \hspace{3mm} [I_n]_{:, \bar{i}^{(j)}},
			&& \mbox{otherwise},
		\end{aligned}
		\right.
	\end{equation}
	where $[I_n]_{:, j}$ is the $j$-th unit vector in $\Rn$.
\end{proposition}

\begin{remark}
	When there exists a column index $j \in \{1, 2, \dotsc, p\}$ such that $\argmin_{i \in \supp ([S]_{:, j})} [\nabla f (Z) - \eta Z]_{i, j}$ is not a singleton and $[W]_{:, j} = 0$, the global minimizer of problem~\eqref{opt:subp-supp} is not unique.
	For this case, the choice of a particular global minimizer does not affect either the algorithmic design or the theoretical analysis.
	In practice, we select the minimal index $\bar{i}^{(j)}$ in $\argmin_{i \in \supp ([S]_{:, j})} [\nabla f (Z) - \eta Z]_{i, j}$ for $[\bar{X}]_{:, j}$.
\end{remark}

\begin{proof}
	On account of the orthogonality of both $X$ and $Z$, we have
	\begin{equation*}
		\begin{aligned}
			\bar{f}_{Z} (X)
			= {} & f (Z) + \jkh{\nabla f (Z), X - Z} + \dfrac{\eta}{2} \norm{X - Z}\ffs \\
			= {} & \jkh{X, \nabla f (Z) - \eta Z} + f (Z) - \jkh{\nabla f (Z), Z} + \eta p.
		\end{aligned}
	\end{equation*}
	Upon omitting constant terms, problem \eqref{opt:subp-supp} further simplifies to the following optimization model,
	\begin{equation*}
		%\label{opt:subp-supp}
		\begin{aligned}
			\min_{X \in \Opnp} \hspace{2mm} & \jkh{X, \nabla f (Z) - \eta Z} \\
			\st \hspace{3mm} & \supp (X) \subseteq \supp (S).
		\end{aligned}
	\end{equation*}
	A straightforward verification reveals that the above problem is separable with respect to column vectors.
	In fact, the optimization problem for the $j$-th column can be formulated as
	\begin{equation}
		\label{opt:subp-supp-j}
		\begin{aligned}
			\min_{x \in \Rn} \hspace{2mm} & \jkh{x, [\nabla f (Z) - \eta Z]_{:, j}} \\
			\st \hspace{1.5mm} & x \geq 0, \; \norm{x}_2 = 1, \; \supp (x) \subseteq \supp ([S]_{:, j}).
		\end{aligned}
	\end{equation}
	For convenience, we define $b_{j} = [W]_{:, j} = \max \{0, [(\eta Z - \nabla f (Z)) \odot S]_{:, j}\} \in \Rn_{+}$.
	It is clear that the solution given by \eqref{eq:sol-supp} satisfies all constraints of problem \eqref{opt:subp-supp-j}.
	Let $x \in \Rn$ be an arbitrary feasible point of problem \eqref{opt:subp-supp-j}.
	If $b_{j} \neq 0$, we have $[(\nabla f (Z) - \eta Z) \odot S]_{:, j} = a_{j} - b_{j}$, where $a_{j} = \max \{0, [(\nabla f (Z) - \eta Z) \odot S]_{:, j}\} \in \Rn_{+}$.
	%	\begin{equation*}
		%		[(\nabla f (Z) - \eta Z) \odot S]_{:, j} = a_{j} - b_{j},
		%	\end{equation*}
	Then it follows from the relationship $\supp (x) \subseteq \supp ([S]_{:, j})$ that
	\begin{equation*}
		\begin{aligned}
			\jkh{x, [\nabla f (Z) - \eta Z]_{:, j}}
			= {} & \jkh{x, [(\nabla f (Z) - \eta Z) \odot S]_{:, j}} \\
			= {} & \jkh{x, a_{j} - b_{j}}
			\geq - \jkh{x, b_{j}}
			\geq - \norm{b_{j}}_2,
		\end{aligned}
	\end{equation*}
	where the equality is achieved at $x = b_{j} / \norm{b_{j}}_2$.
	Otherwise, if $b_{j} = 0$, it holds that $[\nabla f (Z) - \eta Z]_{i, j} \geq 0$ for all $i \in \supp ([S]_{:, j})$.
	Hence, we can obtain that
	\begin{equation*}
		\begin{aligned}
			\jkh{x, [\nabla f (Z) - \eta Z]_{:, j}}
			= {} & \sum_{i \in \supp ([S]_{:, j})} [\nabla f (Z) - \eta Z]_{i, j} [x]_{i} \\
			%& \geq \sum_{i \in \supp ([S]_{:, j})} [\nabla f (Z) - \eta Z]_{i, j} [x]_{i}^2
			\geq {} & [\nabla f (Z) - \eta Z]_{\bar{i}^{(j)}, j} \sum_{i \in \supp ([S]_{:, j})} [x]_{i}^2
			= [\nabla f (Z) - \eta Z]_{\bar{i}^{(j)}, j},
		\end{aligned}
	\end{equation*}
	where the equality is attained at $x = [I_n]_{:, \bar{i}^{(j)}}$.
	The proof is completed.
\end{proof}

The foregoing proposition establishes that the global minimizer of problem~\eqref{opt:subp-supp} possesses a closed-form expression, which can be computed with negligible computational effort.
In general, this amounts merely to normalizing the columns of the matrix $W = \max \{0, ( \eta Z - \nabla f (Z) ) \odot S\}$, a procedure with complexity $O (n)$ as $W$ has at most $n$ nonzero entries.
Beyond this, the computation of $\nabla f (Z)$ benefits significantly from the inherent sparsity structure.
On the one hand, it suffices to evaluate the gradient only at the positions specified by the sign matrix $S$, whose cardinality never exceeds $n$.
On the other hand, the matrix computations involved in $\nabla f (Z)$ can be carried out with considerably reduced complexity by exploiting the sparsity of $Z$.
Consequently, solving problem \eqref{opt:subp-supp} offers a promising pathway toward rapidly realizing a pronounced reduction in the objective function value.

\section{Algorithm Development}

\label{sec:algorithm}

The analysis in Section \ref{sec:support} demonstrates that the objective function value in problem \eqref{opt:stplus} can be substantially reduced with relatively low computational cost when the support set is fixed.
The fundamental difficulty remains the identification of a support set superior to the current one.
To attain a solution of higher quality, it becomes essential to explore updates to the support set that can drive further descent.
For this purpose, we propose a tailored support-set algorithm designed to navigate the combinatorial nature of problem \eqref{opt:stplus} while preserving feasibility.

\subsection{Refined Strategy for Zero Rows}

\label{subsec:zero-rows}

Let $X_k \in \Opnp$ be the current iterate of our algorithm at the $k$‑th iteration.
At this stage, our goal is to generate an intermediate iterate $Y_k \in \Opnp$ by minimizing $\bar{f}_{X_k}$ within a suitable support set specified by a sign matrix $S_k \in \sign (\Opnp)$.
The selection of support sets is thus of paramount importance, as it directly influences the quality and efficiency of the overall procedure.
A natural and principled choice is to adopt the support set of $X_k$ itself.
When it contains zero rows, this strategy may fail to produce a first-order stationary point of the original problem \eqref{opt:stplus}, as condition \eqref{eq:sc-zrow} is not necessarily satisfied in this setting.
To address this issue, we develop a procedure to refine the support set of $X_k$.

% more delicate
Our attention is restricted to the situation where $\zrow (X_k) \neq \varnothing$.
Otherwise, we directly set $S_k = \sign (X_k)$.
Let $j_k^{(i)}$ be the minimal column index associated with the smallest entry in the $i$-th row of $\nabla f (X_k)$, namely,
\begin{equation*}
	j_k^{(i)} = \min \hkh{ j\uast \;\middle|\; j\uast \in \argmin_{ j \in \{1, 2, \dotsc, p\} } [\nabla f (X_k)]_{i, j} }.
\end{equation*}
Each zero row $i \in \zrow (X_k)$ can be refined by activating a nonzero entry at the position $(i, j_k^{(i)})$, in a manner that the resulting point $Y_k$ potentially achieves a further reduction in the objective function value.
It will become evident that this particular choice of $j_k^{(i)}$ proves to be critical in the subsequent theoretical developments.
Given the current support set of $X_k$, we construct a corresponding sign matrix $S_k \in \sign (\Opnp)$ to facilitate this update as follows,
\begin{equation}
	\label{eq:sign-zrow}
	\left\{
	\begin{aligned}
		& [S_k]_{i, :} = \sign ([X_k]_{i, :}), \mbox{ for } i \notin \zrow (X_k), \\
		& [S_k]_{i, j_k^{(i)}} = 1, \; [S_k]_{i, j} = 0, \mbox{ for } i \in \zrow (X_k) \mbox{ and } j \neq j_k^{(i)}.
	\end{aligned}
	\right.
\end{equation}
It can be observed that, $S_k$ retains the original positions of nonzero entries in $X_k$, while simultaneously endowing the zero rows of $X_k$ with specific locations to accommodate newly created nonzero entries.
This adjustment of the support set leaves the orthogonality across columns intact.
Then the intermediate iterate $Y_k$ can be obtained by solving the optimization problem below,
\begin{equation}
	\label{opt:subp-zrow}
	\begin{aligned}
		Y_k \in \argmin_{X \in \Opnp} \hspace{2mm} & \bar{f}_{X_k} (X) \\
		\st \hspace{3.5mm} & \supp (X) \subseteq \supp (S_k).
	\end{aligned}
\end{equation}
Then Proposition \ref{prop:supp} guarantees that the $j$-th column of $Y_k$, for all $j \in \{1, 2, \dotsc, p\}$, can be computed explicitly in closed form as follows,
\begin{equation}
	\label{eq:sol-subp-zrow}
	{[Y_k]_{:, j}} =
	\left\{
	\begin{aligned}
		& \dfrac{[W_k]_{:, j}}{\norm{[W_k]_{:, j}}_2},
		&& \mbox{if } [W_k]_{:, j} \neq 0, \\
		& \hspace{3mm} [I_n]_{:, i_k^{(j)}},
		&& \mbox{otherwise},
	\end{aligned}
	\right.
\end{equation}
with
\begin{equation*}
	W_k = \max \hkh{ 0, ( \eta X_k - \nabla f (X_k) ) \odot S_k } \in \Rpnp,
\end{equation*}
and
\begin{equation*}
	i_k^{(j)} = \min \hkh{ i\uast \;\middle|\; i\uast \in \argmin_{ i \in \supp ([S_k]_{:, j}) } [\nabla f (X_k) - \eta X_k]_{i, j} }. 
\end{equation*}
%where $i_k^{(j)} = \min \{ i\uast \mid i\uast \in \argmin_{ i \in \supp ([S_k]_{:, j}) } [\nabla f (X_k) - \eta X_k]_{i, j} \}$ and $W_k = \max \{ S_k \odot ( \eta X_k - \nabla f (X_k) ), 0 \} \in \Rpnp$.
%\begin{equation*}
%	i_k^{(j)} \in \argmin_{ i \in \supp ([S_k]_{:, j}) } [\nabla f (X_k) - \eta X_k]_{i, j}.
%\end{equation*}

We now take a closer look at the newly updated entries of $Y_k$ in the original zero rows.
Consider a zero row $i \in \zrow (X_k)$ that satisfies $[\nabla f (X_k)]_{i, j_k^{(i)}} < 0$.
In this case, since $[W_k]_{:, j_k^{(i)}} \neq 0$, the updated value of $Y_k$ at $(i, j_k^{(i)})$ is
\begin{equation*}
	{[Y_k]_{i, j_k^{(i)}}}
	= - \dfrac{1}{\norm{[W_k]_{:, j_k^{(i)}}}_2} [\nabla f (X_k)]_{i, j_k^{(i)}} > 0.
\end{equation*}
This observation reveals that, for any zero row failing to meet the stationarity condition \eqref{eq:sc-zrow}, a nonzero entry will indeed be introduced at the selected position.
Consequently, our strategy effectively activates the zero row and expands the current support set.
In reverse, the zero row $i \in \zrow (X_k)$ adheres to the stationarity condition \eqref{eq:sc-zrow} automatically if $[\nabla f (X_k)]_{i, j_k^{(i)}} \geq 0$.
A similar argument then indicates that the selected entry at $(i, j_k^{(i)})$ will be updated to either $0$ or $1$, both of which are reasonable outcomes.
On the one hand, retaining such zero rows is acceptable under condition \eqref{eq:sc-zrow}.
On the other hand, updating such entries to $1$ potentially facilitates the exploration of a broader range of support patterns in future iterations.
In either scenario, the objective function value is expected to decrease further, thereby contributing to the overall progress of our algorithm.

\subsection{Update Scheme for Support Sets}

\label{subsec:support-set}

In this subsection, we turn our attention to the construction of the next iterate $X_{k + 1} \in \Opnp$ based on the intermediate iterate $Y_k \in \Opnp$.
The aim of this stage is to find a new support set that promises a substantial reduction in the objective function value.
%identify and switch to
As the iterations proceed, some nonzero entries in particular rows gradually shrink toward zero.
%which signals that the current iterate is approaching the boundary of its effective support.
This phenomenon suggests the potential for pursuing new descent directions by explicitly switching to a new support set.
%traversing this boundary and
Building on this insight, we devise an update scheme to adjust the positions of nonzero entries for such rows.
%It is unnecessary to update every row, which will incur considerable computational overhead.

Let $\delta \in (0, 1)$ be a constant.
We identify the rows of $Y_k$ whose norms do not exceed a prescribed threshold as follows,
\begin{equation}
	\label{eq:srow}
	\srow (Y_k, \delta_k) = \hkh{ i \mid 0 < \norm{[Y_k]_{i, :}}_2 \leq \delta_k, i \in \{1, 2, \dotsc, n\} },
\end{equation}
where $\delta_k = \max\{ \delta, \min\{ [Y_k]_{i, j} \mid (i, j) \in \supp (Y_k)\} \}$.
Since each row of $Y_k$ contains at most a single nonzero entry, the $\ell_2$-norm of $[Y_k]_{i, :}$ precisely coincides with the value of that entry at the $i$-th row.
Accordingly, the set $\srow (Y_k, \delta_k)$ collects the indices of rows whose nonzero entries do not exceed $\delta_k$.
From the definition of $\delta_k$, it directly follows that $\srow (Y_k, \delta_k)$ must contain at least the row corresponding to the smallest nonzero entry of $Y_k$.
The zero rows in $Y_k$ are excluded from $\srow (Y_k, \delta_k)$, as they can be updated by invoking the refined strategy outlined in Section \ref{subsec:zero-rows}.

For each entry in $\srow (Y_k, \delta_k)$, we explore relocating its nonzero entry to other columns and select the column that yields the lowest function value, which in turn delineates the updated support set.
To make the description more precise, we denote the indices of the selected rows by
\begin{equation}
	\label{eq:srow-u}
	\srow (Y_k, \delta_k) = \{u^{(1)}, u^{(2)}, \dotsc, u^{(r_k)}\},
\end{equation}
with $1 \leq r_k \leq n$.
Moreover, the notation $\hat{Y}_{k}^{(t)} \in \Opnp$ stands for the intermediate iterate obtained after updating the first $t$ rows in $\srow (Y_k, \delta_k)$.
Starting from $\hat{Y}_{k}^{(0)} = Y_k$, we sequentially update the positions of nonzero entries in the rows specified by $\srow (Y_k, \delta_k)$ to generate the next iterate $X_{k + 1} \in \Opnp$.

To elucidate our strategy, we take as an example the update of the $u^{(t)}$-th row in $\hat{Y}_{k}^{(t - 1)}$ for $t \in \{1, 2, \dotsc, r_k\}$.
It is noteworthy that, reassigning the positions of nonzero entries equal to $1$ would inevitably result in zero columns within the matrix, which violates the feasibility of $\Opnp$.
Let $(u^{(t)}, w^{(t)})$ be the original position of the nonzero entry in the $u^{(t)}$-th row of $\hat{Y}_{k}^{(t - 1)}$.
If $[\hat{Y}_{k}^{(t - 1)}]_{u^{(t)}, w^{(t)}} = 1$, we simply take $\hat{Y}_{k}^{(t)}$ to be $\hat{Y}_{k}^{(t - 1)}$.
Then our focus is shifted to the situation where $[\hat{Y}_{k}^{(t - 1)}]_{u^{(t)}, w^{(t)}} < 1$.
Our algorithm attempts to relocate the nonzero entry in the $u^{(t)}$-th row of $\hat{Y}_{k}^{(t - 1)}$ to the $v$-th column.
Based on the support set of $\hat{Y}_{k}^{(t - 1)}$, the following sign matrix $\hat{S}_{k}^{(t, v)} \in \sign (\Opnp)$ is constructed accordingly to guide this update,
\begin{equation}
	\label{eq:sign-cl}
	\left\{
	\begin{aligned}
		& [\hat{S}_{k}^{(t, v)}]_{i, :} = \sign ([\hat{Y}_{k}^{(t - 1)}]_{i, :}), \mbox{ for } i \notin \zrow (Y_k) \mbox{ and } i \neq u^{(t)}, \\
		& [\hat{S}_{k}^{(t, v)}]_{i, \hat{j}_k^{(i)}} = 1, \; [\hat{S}_{k}^{(t, v)}]_{i, j} = 0, \mbox{ for } i \in \zrow (Y_k) \mbox{ and } j \neq \hat{j}_k^{(i)}, \\
		& [\hat{S}_{k}^{(t, v)}]_{i, v} = 1, \; [\hat{S}_{k}^{(t, v)}]_{i, j} = 0, \mbox{ for } i = u^{(t)} \mbox{ and } j \neq v,
	\end{aligned}
	\right.
\end{equation}
where, for each $i \in \zrow (Y_k)$, it holds that
\begin{equation}
	\label{eq:hatj}
	\hat{j}_k^{(i)} = \min \hkh{ j\uast \;\middle|\; j\uast \in \argmin_{ j \in \{1, 2, \dotsc, p\} } [\nabla f (Y_k)]_{i, j} }.
\end{equation}
A closer inspection illustrates that, aside from the zero rows of $Y_k$ and the $u^{(t)}$-th row targeted for update, $\hat{S}_{k}^{(t, v)}$ faithfully maintains the positions of nonzero entries in $\hat{Y}_{k}^{(t - 1)}$.
The zero rows of $Y_k$ are treated by the refined strategy introduced in the prior subsection, whereas the nonzero entry of the $u^{(t)}$-th row is reassigned to the $v$-th column.
Furthermore, it follows from the condition $[\hat{Y}_{k}^{(t - 1)}]_{u^{(t)}, w^{(t)}} < 1$ that $\hat{S}_{k}^{(t, v)}$ corresponds to the sign pattern of an element in $\Opnp$.

For the purpose of generating a candidate of the next iterate, we then proceed to minimize $\bar{f}_{Y_k}$ within the support set specified by $\hat{S}_{k}^{(t, v)}$ as follows,
\begin{equation}
	\label{opt:subp-cl}
	\begin{aligned}
		\hat{Y}_{k}^{(t, v)} \in \argmin_{X \in \Opnp} \hspace{2mm} & \bar{f}_{Y_k} (X) \\
		\st \hspace{3.5mm} & \supp (X) \subseteq \supp (\hat{S}_{k}^{(t, v)}).
	\end{aligned}
\end{equation}
By invoking Proposition \ref{prop:supp}, we know that the $j$-th column of $\hat{Y}_{k}^{(t, v)}$, for all $j \in \{1, 2, \dotsc, p\}$, admits the following closed-form expression,
\begin{equation}
	\label{eq:sol-subp-cl}
	[\hat{Y}_{k}^{(t, v)}]_{:, j} =
	\left\{
	\begin{aligned}
		& \dfrac{[\hat{W}_k^{(t, v)}]_{:, j}}{\norm{[\hat{W}_k^{(t, v)}]_{:, j}}_2},
		&& \mbox{if } [\hat{W}_k^{(t, v)}]_{:, j} \neq 0, \\
		& \hspace{3mm} [I_n]_{:, \hat{i}_k^{(j)}},
		&& \mbox{otherwise},
	\end{aligned}
	\right.
\end{equation}
with
\begin{equation*}
	\hat{W}_k^{(t, v)} = \max \hkh{ 0, ( \eta Y_k - \nabla f (Y_k) ) \odot \hat{S}_k^{(t, v)} } \in \Rpnp,
\end{equation*}
and
\begin{equation*}
	\hat{i}_k^{(j)} = \min \hkh{ i\uast \;\middle|\; i\uast \in \argmin_{ i \in \supp ([\hat{S}_k^{(t, v)}]_{:, j}) } [\nabla f (Y_k) - \eta Y_k]_{i, j} }.
\end{equation*}
%where $\hat{i}_k^{(j)} = \min \{ i\uast \mid i\uast \in \argmin_{ i \in \supp ([\hat{S}_k^{(t, v)}]_{:, j}) } [\nabla f (Y_k) - \eta Y_k]_{i, j} \}$ and $\hat{W}_k^{(t, v)} = \max \{0, ( \eta Y_k - \nabla f (Y_k) ) \odot \hat{S}_k^{(t, v)}\} \in \Rpnp$.
%\begin{equation*}
%	\hat{i}_k^{(j)} \in \argmin_{ i \in \supp ([\hat{S}_k^{(t, v)}]_{:, j}) } [\nabla f (Y_k) - \eta Y_k]_{i, j}.
%\end{equation*}

It is worth emphasizing that the global minimizer $\hat{Y}_{k}^{(t, v)}$ obtained from subproblem \eqref{opt:subp-cl} does not necessarily lead to a reduction in the function value of $\bar{f}_{Y_k}$.
In fact, it is possible that
\begin{equation*}
	\bar{f}_{Y_k} (\hat{Y}_{k}^{(t, v)}) > \bar{f}_{Y_k} (\hat{Y}_{k}^{(t - 1)}),
\end{equation*}
as the support set of $\hat{Y}_{k}^{(t, v)}$ may be distinct from that of $\hat{Y}_{k}^{(t - 1)}$.
To mitigate this issue, we exhaustively explore all possible target columns $v \in \{1, 2, \dotsc, p\}$ and solve subproblem \eqref{opt:subp-cl} for each case.
Among the resulting candidates, we identify the one that yields the lowest function value of $\bar{f}_{Y_k}$, denoted by $\hat{Y}_{k}^{(t)} = \hat{Y}_{k}^{(t, v^{(t)})}$ with
\begin{equation}
	\label{eq:yk-ast}
	v^{(t)} = \min \hkh{ v\uast \;\middle|\; v\uast \in \argmin_{ v \in \{1, 2, \dotsc, p\} } \bar{f}_{Y_k} (\hat{Y}_{k}^{(t, v)}) }.
\end{equation}
Let us recall that $w^{(t)}$ is the column index in which the nonzero entry of the $u^{(t)}$-th row originally resides.
Then it follows from the definition of $\hat{Y}_{k}^{(t)}$ that
\begin{equation}
	\label{eq:haty-t-ast}
	\bar{f}_{Y_k} (\hat{Y}_{k}^{(t)})
	= \bar{f}_{Y_k} (\hat{Y}_{k}^{(t, v^{(t)})})
	\leq \bar{f}_{Y_k} (\hat{Y}_{k}^{(t, w^{(t)})})
	\leq \bar{f}_{Y_k} (\hat{Y}_{k}^{(t - 1)}),
\end{equation}
where the second inequality holds since $\hat{Y}_{k}^{(t - 1)}$ is also feasible for subproblem \eqref{opt:subp-cl} with $v = w^{(t)}$.
It may happen that $v^{(t)} = w^{(t)}$, which indicates that the current configuration of the $u^{(t)}$-th row is retained without modification.
Once all $r_k$ rows in $\srow (Y_k, \delta_k)$ have been updated, we obtain the next iterate $X_{k + 1} = \hat{Y}_{k}^{(r_k)}$.

Within the framework of the above construction, the next iterate $X_{k + 1} \in \Opnp$ can be interpreted as the optimal solution of the following subproblem,
\begin{equation}
	\label{opt:subp-star}
	\begin{aligned}
		X_{k + 1} \in \argmin_{X \in \Opnp} \hspace{2mm} & \bar{f}_{Y_k} (X) \\
		\st \hspace{3.5mm} & \supp (X) \subseteq \supp (\hat{S}_{k}),
	\end{aligned}
\end{equation}
where the sign matrix $\hat{S}_{k} \in \sign (\Opnp)$ is given by
\begin{equation*}
	\left\{
	\begin{aligned}
		& [\hat{S}_{k}]_{i, :} = \sign ([Y_k]_{i, :}), \mbox{ for } i \notin \zrow (Y_k) \mbox{ and } i \notin \srow (Y_k, \delta_k), \\
		& [\hat{S}_{k}]_{i, \hat{j}_k^{(i)}} = 1, \; [\hat{S}_{k}]_{i, j} = 0 \mbox{ for } i \in \zrow (Y_k) \mbox{ and } j \neq \hat{j}_k^{(i)}, \\
		& [\hat{S}_{k}]_{u^{(t)}, v^{(t)}} = 1, \; [\hat{S}_{k}]_{u^{(t)}, j} = 0, \mbox{ for } t \in \{1, 2, \dotsc, r_k\} \mbox{ and } j \neq v^{(t)}.
	\end{aligned}
	\right.
\end{equation*}
The explicit expression of $X_{k + 1}$ can be derived in a manner akin to that of \eqref{eq:sol-subp-cl}, and it is omitted here for the sake of brevity.

\subsection{Complete Framework}

This subsection integrates the processes described in the preceding parts to formulate a complete algorithmic framework for problem \eqref{opt:stplus}.
We refer to it as \textit{support-set algorithm}, which is denoted by \SUPPORT.

In practice, it is often unnecessary to update the support set with high frequency.
Instead, we only need to do so when the objective function value fails to exhibit adequate descent within the current support set.
This selective strategy mitigates the risk of switching to suboptimal support sets that may yield higher objective function values, thereby substantially reducing computational burden.
Lemma \ref{le:lip-f} unveils that the reduction in the objective function value, from $f (X_k)$ to $f (Y_k)$, is proportional to $\norm{Y_k - X_k}\ffs$.
Accordingly, we determine whether to preserve the support set or not by comparing $\norm{Y_k - X_k}\ff$ with a prescribed constant $\theta > 0$.
If $\norm{Y_k - X_k}\ff \geq \theta$, it suggests that the objective function value can still achieve sufficient descent within the current support set, which is thus retained.
In this case, we directly set $X_{k + 1} = Y_k$.
Otherwise, the support set is updated to facilitate further progress.
We adopt the procedure described in Section \ref{subsec:support-set} to generate the next iterate $X_{k + 1}$.
Algorithm~\ref{alg:supps} outlines the complete framework of \SUPPORT for problem \eqref{opt:stplus} with $1 < p < n$.

\begin{algorithm2e}[ht]
	%\SetAlgoLined
	\caption{support-set algorithm (\SUPPORT).}
	\label{alg:supps}
	
	\KwIn{$X_0 \in \Opnp$, $\eta > L$, $\delta \in (0, 1)$, and $\theta > 0$.}
	
	%Initialize $X_{0} = X_{\init}$.
	
	\For{$k = 0, 1, 2, \dotsc$}{
		
		Generate the sign matrix $S_k \in \sign (\Opnp)$ by \eqref{eq:sign-zrow}.
		
		Update $Y_k \in \Opnp$ by \eqref{eq:sol-subp-zrow}.
		
		\If{$\norm{Y_k - X_k}\ff \geq \theta$}{
			
			Set $X_{k + 1} = Y_k$.
			
		}
		
		\Else{
			
			Compute $\delta_k = \max\{ \delta, \min\{ [Y_k]_{i, j} \mid (i, j) \in \supp (Y_k) \} \}$.
			
			Identify $\{u^{(1)}, u^{(2)}, \dotsc, u^{(r_k)}\}$ by \eqref{eq:srow} and \eqref{eq:srow-u}.
			
			Set $\hat{Y}_{k}^{(0)} = Y_k$.
			
			\For{$t = 1, 2, \dotsc, r_k$}{
				
				\If{$[\hat{Y}_{k}^{(t - 1)}]_{u^{(t)}, w^{(t)}} = 1$}{
					
					Set $\hat{Y}_{k}^{(t)} = \hat{Y}_{k}^{(t - 1)}$.
					
				}
				
				\Else{
					
					\For{$v = 1, 2, \dotsc, p$}{
						
						Generate the sign matrix $\hat{S}_{k}^{(t, v)} \in \sign (\Opnp)$ by \eqref{eq:sign-cl}.
						
						Update $\hat{Y}_{k}^{(t, v)} \in \Opnp$ by \eqref{eq:sol-subp-cl}.
						
					}
					
					Choose $\hat{Y}_{k}^{(t)} = \hat{Y}_{k}^{(t, v^{(t)})}$ by \eqref{eq:yk-ast}.
					
				}
				
			}
			
			Set $X_{k + 1} = \hat{Y}_{k}^{(r_k)}$.
			
		}
		
	}
	
	\KwOut{$X_{k + 1}$.}
	
\end{algorithm2e}

The computational overhead of a single iteration in \SUPPORT is exceedingly low.
As previously discussed, subproblems \eqref{opt:subp-zrow} and \eqref{opt:subp-cl} both can be solved with a cost of merely $O (n)$.
During the update scheme of support sets at the $k$-th iteration, one needs to tackle subproblem \eqref{opt:subp-cl} a total of $p$ times for each of the $r_k$ rows in $\srow (Y_k, \delta_k)$.
At first glance, this procedure might appear computationally demanding; however, this is not the case.
Specifically, Proposition \ref{prop:supp} asserts that both the global minimizer and the optimal value are completely determined by the matrix $\hat{W}_k^{(t, v)}$ for subproblem \eqref{opt:subp-cl}.
Indeed, for any $t_1 \neq t_2$ and $v_1 \neq v_2$, the associated matrices $\hat{W}_k^{(t_1, v_1)}$ and $\hat{W}_k^{(t_2, v_2)}$ differ in only four entries, a structural property that enables a highly efficient implementation of this step.
Consequently, it suffices to compute the optimal value of subproblem \eqref{opt:subp-cl} in full detail once---say, for $t = 1$ and $v = 1$.
The computation in all subsequent cases with $t \neq 1$ and $v \neq 1$ just involves the update of four differing entries.
As a result, solving all $p$ instances of subproblem \eqref{opt:subp-cl} for $t \neq 1$ incurs a total computational cost of only $O (p)$.
Now we can conclude that the overall computational complexity of the $k$-th iteration is at most $O(n + r_k p)$.
In sharp contrast, existing algorithms \cite{Jiang2023exact,Qian2024error} require computing the projections onto $\Qpnp$ or $\Onp$ per iteration, with the corresponding computational costs amounting to $O (np)$ or $O (np^2)$, respectively.
Furthermore, the computational burden entailed by gradient evaluations in \SUPPORT remains modest thanks to the inherent sparsity structure.

\section{Convergence Analysis}

\label{sec:convergence}

This section delves into the convergence analysis of the proposed algorithm.
Specifically, any accumulation point of the sequence generated by Algorithm \ref{alg:supps} is shown to be a first-order stationary point.
We also provide the iteration complexity to reach an approximate first-order stationary point.
A noteworthy property of finite support identification is then established for our algorithm.
Finally, under an additional semi-algebraicity assumption, we further derive the whole-sequence convergence of the generated sequence.

\subsection{Auxiliary Results}

In this subsection, we present a collection of auxiliary and preparatory results, which serve as the foundation for the subsequent convergence analysis.

The following lemma first shows that the sequence $\{f (X_{k})\}$ of function values exhibits a sufficient descent property.

\begin{lemma}
	\label{le:des-f}
	Let $\{(X_k, Y_k)\}$ be the sequence generated by Algorithm \ref{alg:supps}.
	Then, for all $k \in \bN$, it follows that
	\begin{equation}
		\label{eq:des-f}
		f (X_k) - f (X_{k + 1})
		\geq \dfrac{\eta - L}{2} \norm{Y_k - X_k}\ffs
		+ \dfrac{\eta - L}{2} \norm{X_{k + 1} - Y_k}\ffs.
	\end{equation}
\end{lemma}

\begin{proof}
	From the construction of the sign matrix $S_k$ in \eqref{eq:sign-zrow}, we can obtain that $\supp (X_k) \subseteq \supp (S_k)$, which indicates that $X_k$ is a feasible point of subproblem \eqref{opt:subp-zrow}.
	Then the global optimality of $Y_k$ implies that $\bar{f}_{X_k} (Y_k) \leq \bar{f}_{X_k} (X_k)$.
	%	\begin{equation*}
		%		\bar{f}_{X_k} (Y_k) \leq \bar{f}_{X_k} (X_k).
		%	\end{equation*}
	As a direct consequence of Lemma \ref{le:lip-f}, we can proceed to show that
	\begin{equation}
		\label{eq:des-f-xk}
		f (X_k) - f (Y_k) \geq \dfrac{\eta - L}{2} \norm{Y_k - X_k}\ffs.
	\end{equation}
	The update scheme of $\hat{Y}_{k}^{(t)}$ indicates that either $\hat{Y}_{k}^{(t)} = \hat{Y}_{k}^{(t - 1)}$ or it satisfies the relationship \eqref{eq:haty-t-ast}.
	In both cases, it holds that $\bar{f}_{Y_k} (\hat{Y}_{k}^{(t)}) \leq \bar{f}_{Y_k} (\hat{Y}_{k}^{(t - 1)})$.
	%	\begin{equation*}
		%		\bar{f}_{Y_k} (\hat{Y}_{k}^{(t)})
		%		\leq \bar{f}_{Y_k} (\hat{Y}_{k}^{(t - 1)}).
		%	\end{equation*}
	By applying this relationship recursively for $r_k$ successive steps, we readily arrive at
	\begin{equation*}
		\bar{f}_{Y_k} (X_{k + 1})
		= \bar{f}_{Y_k} (\hat{Y}_{k}^{(r_k)})
		\leq \bar{f}_{Y_k} (\hat{Y}_{k}^{(0)})
		= \bar{f}_{Y_k} (Y_k).
	\end{equation*}
	Similarly, it follows from Lemma \ref{le:lip-f} that
	\begin{equation}
		\label{eq:des-f-yk}
		f (Y_k) - f (X_{k + 1}) \geq \dfrac{\eta - L}{2} \norm{X_{k + 1} - Y_k}\ffs.
	\end{equation}
	Now we can obtain the assertion \eqref{eq:des-f} of this lemma by collecting two relationships \eqref{eq:des-f-xk} and \eqref{eq:des-f-yk} together.
	The proof is completed.
\end{proof}

As an immediate corollary of Lemma \ref{le:des-f}, we proceed to establish that the distance between two consecutive iterates generated by Algorithm \ref{alg:supps} converges to zero.

\begin{corollary}
	\label{coro:dist-iter}
	Let $\{(X_k, Y_k)\}$ be the sequence generated by Algorithm \ref{alg:supps}.
	Then it holds that
	\begin{equation}
		\label{eq:limit-x}
		\lim_{k \to \infty} \norm{Y_k - X_k}\ffs + \norm{X_{k + 1} - Y_k}\ffs = 0.
	\end{equation}
\end{corollary}

\begin{proof}
	Since $\Onp$ is a compact manifold and $f$ is continuous over $\Onp$, there exist two constants $\underline{f}$ and $\overline{f}$ such that
	\begin{equation}
		\label{eq:bound-f}
		\underline{f} \leq f (X) \leq \overline{f},
	\end{equation}
	for any $X \in \Onp$.
	Summing the relationship \eqref{eq:des-f} over $k$ from $0$ to $K - 1$ results in that
	\begin{equation*}
		f (X_0) - f (X_K)
		\geq \dfrac{\eta - L}{2} \sum_{k = 0}^{K - 1} \dkh{ \norm{Y_k - X_k}\ffs + \norm{X_{k + 1} - Y_k}\ffs },
	\end{equation*}
	which further implies that
	\begin{equation}
		\label{eq:sum-norm}
		\sum_{k = 0}^{K - 1} \dkh{ \norm{Y_k - X_k}\ffs + \norm{X_{k + 1} - Y_k}\ffs }
		\leq \dfrac{2}{\eta - L} \dkh{f (X_0) - f (X_K)}
		\leq \dfrac{2 (\overline{f} - \underline{f})}{\eta - L}.
	\end{equation}
	Passing to the limit $K \to \infty$ in \eqref{eq:sum-norm} immediately yields the conclusion asserted in \eqref{eq:limit-x}.
	We complete the proof.
\end{proof}

Another important result reveals that the stationarity violation can be controlled in terms of the distance between consecutive iterates, as articulated in the proposition below.

\begin{proposition}
	\label{prop:fosc}
	Let $\{(X_k, Y_k)\}$ be the sequence generated by Algorithm \ref{alg:supps}.
	Then the following relationship is satisfied,
	\begin{equation}
		\label{eq:supp-yk}
		\abs{[\grad\, f (Y_k)]_{i, j}} \leq 2 (\eta + L) \norm{Y_k - X_k}\ff,
	\end{equation}
	for all $(i, j) \in \supp (Y_k)$.
	Moreover, if $\norm{Y_k - X_k}\ff < \theta$, there exists a constant $M > 0$ such that
	\begin{equation}
		\label{eq:zrow-yk}
		\max\{0, - [\nabla f (Y_k)]_{i, j}\} \leq (\eta + M) \norm{X_{k + 1} - Y_k}\ff,
	\end{equation}
	for all $i \in \zrow (Y_k)$ and $j \in \{1, 2, \dotsc, p\}$.
\end{proposition}

\begin{proof}
	The first purpose is to show that the following relationship holds for all $i \in \supp ([Y_k]_{:, j})$ and $j \in \{1, 2, \dotsc, p\}$,
	\begin{equation}
		\label{eq:sc-yk}
		[\eta X_k - \nabla f (X_k)]_{i, j} - \dkh{[Y_k]_{:, j}\zz [\eta X_k - \nabla f (X_k)]_{:, j}} [Y_k]_{i, j} = 0.
	\end{equation}
	For the $j$-th column satisfying $[W_k]_{:, j} = 0$, we have $\supp ([Y_k]_{:, j}) = i_k^{(j)}$.
	It can be readily verified that the relationship \eqref{eq:sc-yk} holds for $i =  i_k^{(j)}$ based on the closed-form expression of $Y_k$ given in \eqref{eq:sol-subp-zrow}.
	Then we consider the scenario where $[W_k]_{:, j} \neq 0$.
	For all $i \in \supp ([Y_k]_{:, j}) = \supp ([W_k]_{:, j})$, it holds that $[Y_k]_{i, j} = [W_k]_{i, j} / \norm{[W_k]_{:, j}}_2$ and $[W_k]_{i, j} = [\eta X_k - \nabla f (X_k)]_{i, j} > 0$.
	By straightforward calculations, we can obtain that
	\begin{equation*}
		\begin{aligned}
			{[Y_k]_{:, j}}\zz [\eta X_k - \nabla f (X_k)]_{:, j}
			= {} & \sum_{i \in \supp ([Y_k]_{:, j})} [Y_k]_{i, j} [\eta X_k - \nabla f (X_k)]_{i, j} \\
			= {} & \dfrac{1}{\norm{[W_k]_{:, j}}_2} \sum_{i \in \supp ([W_k]_{:, j})} [W_k]_{i, j}^2
			= \norm{[W_k]_{:, j}}_2.
		\end{aligned}
	\end{equation*}
	The above equality directly implies that the relationship \eqref{eq:sc-yk} holds for all $i \in \supp ([Y_k]_{:, j})$.
	
	Next, we proceed to prove that $Y_k$ satisfies the condition \eqref{eq:supp-yk}.
	Let $(i, j) \in \supp (Y_k)$.
	According to the relationship \eqref{eq:sc-yk}, it follows that
	\begin{equation*}
		\begin{aligned}
			{[\grad\, f (Y_k)]_{i, j}}
			= {} & [\nabla f (Y_k)]_{i, j} - \dkh{[Y_k]_{:, j}\zz [\nabla f (Y_k)]_{:, j}} [Y_k]_{i, j} \\
			= {} & [\nabla f (Y_k)]_{i, j}
			- \dkh{[Y_k]_{:, j}\zz [\nabla f (Y_k)]_{:, j}} [Y_k]_{i, j}
			+ [\eta X_k - \nabla f (X_k)]_{i, j} \\
			& - \dkh{[Y_k]_{:, j}\zz [\eta X_k - \nabla f (X_k)]_{:, j}} [Y_k]_{i, j} \\
			%			= {} & [\nabla f (Y_k) - \nabla f (X_k)]_{i, j} + \eta [X_k]_{i, j} \\
			%			& + \dkh{ [Y_k]_{:, j}\zz \dkh{[\nabla f (X_k) - \nabla f (Y_k) + \eta (Y_k - X_k) - \eta Y_k]_{:, j}} } [Y_k]_{i, j} \\
			= {} & [\nabla f (Y_k) - \nabla f (X_k)]_{i, j} - \eta [Y_k - X_k]_{i, j} \\
			& + \dkh{ [Y_k]_{:, j}\zz \dkh{[\nabla f (X_k) - \nabla f (Y_k) + \eta (Y_k - X_k)]_{:, j}} } [Y_k]_{i, j},
		\end{aligned}
	\end{equation*}
	which together with the Lipschitz continuity of $\nabla f$ yields that
	\begin{equation*}
		\begin{aligned}
			\abs{[\grad\, f (Y_k)]_{i, j}}
			\leq {} & \abs{[\nabla f (Y_k) - \nabla f (X_k)]_{i, j}} + \eta \abs{[Y_k - X_k]_{i, j}} \\
			& + \abs{ [Y_k]_{:, j}\zz \dkh{[\nabla f (X_k) - \nabla f (Y_k) + \eta (Y_k - X_k)]_{:, j}} } \\
			\leq {} & \abs{[\nabla f (Y_k) - \nabla f (X_k)]_{i, j}} + \eta \abs{[Y_k - X_k]_{i, j}} \\
			& + \norm{[\nabla f (X_k) - \nabla f (Y_k) + \eta (Y_k - X_k)]_{:, j}}_2 \\
			\leq {} & 2 (\eta + L) \norm{Y_k - X_k}\ff.
		\end{aligned}
	\end{equation*}
	Thus, the relationship \eqref{eq:supp-yk} is satisfied for all $(i, j) \in \supp (Y_k)$.
	
	Finally, we consider an arbitrary zero row $i \in \zrow (Y_k)$ when $\norm{Y_k - X_k}\ff < \theta$.
	If $[\nabla f (Y_k)]_{i, \hat{j}_k^{(i)}} \geq 0$, the iterate $Y_k$ adheres to the condition \eqref{eq:zrow-yk} automatically.
	Then our attention is confined to the case where $[\nabla f (Y_k)]_{i, \hat{j}_k^{(i)}} < 0$.
	Let $\hat{W}_k = \max \{0, ( \eta Y_k - \nabla f (Y_k) ) \odot \hat{S}_k\} \in \Rpnp$.
	As a direct consequence of Proposition \ref{prop:supp} and the relationship \eqref{opt:subp-star}, we can show that
	\begin{equation*}
		[X_{k + 1}]_{i, \hat{j}_k^{(i)}}
		= - \dfrac{1}{\norm{[\hat{W}_k]_{:, \hat{j}_k^{(i)}}}_2} [\nabla f (Y_k)]_{i, \hat{j}_k^{(i)}} > 0.
	\end{equation*}
	Since $f$ is continuously differentiable over the compact manifold $\Onp$, there exists a constant $M > 0$ such that $\norm{\nabla f (X)}\ff \leq M$ for all $X \in \Onp$.
	Hence, we can obtain that
	\begin{equation*}
		\begin{aligned}
			\norm{[\hat{W}_k]_{:, \hat{j}_k^{(i)}}}_2
			= {} & \norm{\max \hkh{ 0, [( \eta Y_k - \nabla f (Y_k) ) \odot \hat{S}_k]_{:, \hat{j}_k^{(i)}} }}_2
			\leq \norm{ [\eta Y_k - \nabla f (Y_k)]_{:, \hat{j}_k^{(i)}} }_2 \\
			\leq {} & \eta \norm{ [Y_k]_{:, \hat{j}_k^{(i)}} }_2
			+ \norm{ [\nabla f (Y_k)]_{:, \hat{j}_k^{(i)}} }_2
			\leq \eta + M,
		\end{aligned}
	\end{equation*}
	which further implies that
	\begin{equation*}
		\norm{X_{k + 1} - Y_k}\ff
		\geq \abs{[X_{k + 1} - Y_k]_{i, \hat{j}_k^{(i)}}}
		= \dfrac{- [\nabla f (Y_k)]_{i, \hat{j}_k^{(i)}}}{\norm{[\hat{W}_k]_{:, \hat{j}_k^{(i)}}}_2}
		\geq \dfrac{- [\nabla f (Y_k)]_{i, \hat{j}_k^{(i)}}}{\eta + M}.
	\end{equation*}
	According to the definition of $\hat{j}_k^{(i)}$ in \eqref{eq:hatj}, it then follows that
	\begin{equation*}
		\max\{0, - [\nabla f (Y_k)]_{i, j}\}
		\leq - [\nabla f (Y_k)]_{i, \hat{j}_k^{(i)}}
		\leq (\eta + M) \norm{X_{k + 1} - Y_k}\ff,
	\end{equation*}
	for all $j \in \{1, 2, \dotsc, p\}$.
	Therefore, we can conclude that the relationship \eqref{eq:zrow-yk} holds.
	The proof is completed.
\end{proof}

\subsection{Subsequence Convergence}

Building upon the auxiliary results derived in the preceding subsection, we proceed to establish the subsequence convergence of Algorithm~\ref{alg:supps} to a first-order stationary point of problem \eqref{opt:stplus}.
From the construction of our algorithm, it is evident that the generated sequence is feasible within $\Opnp$.

\begin{theorem}
	\label{thm:convergence}
	Any accumulation point of the sequence $\{X_k\}$ generated by Algorithm \ref{alg:supps} qualifies as a first-order stationary point of problem \eqref{opt:stplus}.
\end{theorem}

\begin{proof}
	Due to the compactness of $\Opnp$, we know that the sequence $\{X_{k}\}$ is bounded.
	Then from the Bolzano-Weierstrass theorem, it can be deduced that it has at least one accumulation point.
	Let $X\last$ be an accumulation point of $\{X_{k}\}$.
	The closedness of $\Opnp$ guarantees that $X\last \in \Opnp$.
	For notational simplicity, we continue to denote by $\{X_k\}$ the subsequence converging to $X\last$.
	And it follows from Corollary~\ref{coro:dist-iter} that
	\begin{equation*}
		\lim_{k \to \infty} X_{k + 1} = \lim_{k \to \infty} Y_k = \lim_{k \to \infty} X_k = X\last.
	\end{equation*}
	To complete the proof, we now turn our attention to verifying that $X\last$ fulfills conditions \eqref{eq:sc-supp} and \eqref{eq:sc-zrow} simultaneously.
	
	Since $Y_k \in \Opnp$, it has at most $n$ nonzero entries.
	In view of the relationship \eqref{eq:supp-yk}, it can be readily verified that
	\begin{equation}
		\label{eq:yk-odot-grad}
		\begin{aligned}
			\norm{Y_k \odot \grad\, f (Y_k)}\ffs
			= {} & \sum_{(i, j) \in \supp (Y_k)} [Y_k]_{i, j}^2 [\grad\, f (Y_k)]_{i, j}^2 \\
			\leq {} & \sum_{(i, j) \in \supp (Y_k)} [\grad\, f (Y_k)]_{i, j}^2
			\leq 4 n (\eta + L)^2 \norm{Y_k - X_k}\ffs.
		\end{aligned}
	\end{equation}
	Upon taking $k \to \infty$ in \eqref{eq:yk-odot-grad}, we immediately arrive at
	\begin{equation*}
		X\last \odot \grad\, f (X\last) = 0,
	\end{equation*}
	which indicates that $X\last$ adheres to condition \eqref{eq:sc-supp}.
	
	Next, we consider the case where $\zrow (X\last) \neq \varnothing$ and show that condition \eqref{eq:sc-zrow} is satisfied.
	Let $i\last \in \zrow (X\last)$ and $\bK = \{k \in \bN \mid \norm{[Y_k]_{i\last, :}}_2 = 0\}$ be an index set.
	If $\bK$ is infinite, it is clear that $\lim_{\bK \ni k \to \infty} Y_k = X\last$.
	By passing to the limit $\bK \ni k \to \infty$ in \eqref{eq:zrow-yk}, we have
	\begin{equation*}
		\max \hkh{ 0, - [\nabla f (X\last)]_{i\last, j} } = 0,
	\end{equation*}
	for all $j \in \{1, 2, \dotsc, p\}$.
	The above relationship guarantees that $X\last$ satisfies condition \eqref{eq:sc-zrow}.
	Our analysis henceforth centers on the scenario where $\bK$ is finite.
	We assume on the contrary that condition \eqref{eq:sc-zrow} does not hold for $i\last \in \zrow (X\last)$.
	Hence, there exist $b \in \{1, 2, \dotsc, p\}$ and $\tau > 0$ such that
	\begin{equation*}
		[\nabla f (X\last)]_{i\last, b} = - \tau.
	\end{equation*}
	Let $\omega > 0$ be a constant defined as
	\begin{equation*}
		\omega = \min \hkh{ \dfrac{1}{2}, \, \delta, \, \dfrac{\tau}{2 (\eta + M)}, \, \dfrac{\tau^2}{8 (\eta + M)^2} }.
	\end{equation*}
	Since $\bK$ is finite, there exists $k \in \bN$ such that
	\begin{equation*}
		[\nabla f (Y_{k})]_{i\last, b} \leq - \dfrac{\tau}{2},
		\quad
		0 < \norm{[Y_{k}]_{i\last, :}}_2 < \omega,
		\quad\mbox{and}\quad
		0 < \norm{[\hat{Y}_{k}^{(s)}]_{i\last, :}}_2 < \omega,
	\end{equation*}
	for all $1 \leq s \leq r_{k}$.
	From the definition of $\omega$, we know that $i\last \in \srow (Y_k, \delta_k)$.
	Suppose that $i\last$ is the $t$-th element in $\srow (Y_k, \delta_k)$ and $l = v^{(t)}$.
	In the subsequent discussion, we will show that $\bar{f}_{Y_{k}} (\hat{Y}_{k}^{(t, l)}) > \bar{f}_{Y_{k}} (\hat{Y}_{k}^{(t, b)})$, leading to a contradiction with the definition of $v^{(t)}$ in \eqref{eq:yk-ast}.
	
	According to Proposition \ref{prop:supp}, it then follows that $\hat{Y}_{k}^{(t)} = \hat{Y}_{k}^{(t, l)}$, $[\hat{Y}_{k}^{(t, l)}]_{i\last, l} \in (0, \omega)$, and $[\hat{Y}_{k}^{(t, l)}]_{:, l} = [\hat{W}_{k}^{(t, l)}]_{:, l} / \| [\hat{W}_{k}^{(t, l)}]_{:, l} \|_2$.
	If $l = b$, we have
	\begin{equation*}
		[\hat{Y}_{k}^{(t, l)}]_{i\last, l}
		= [\hat{Y}_{k}^{(t, b)}]_{i\last, b}
		= \dfrac{\eta [Y_{k}]_{i\last, b} - [\nabla f (Y_{k})]_{i\last, b}}{\norm{[\hat{W}_{k}^{(t, b)}]_{:, b}}_2}
		\geq \dfrac{\tau}{2 (\eta + M)}
		\geq \omega,
	\end{equation*}
	which is in direct conflict with the fact that $[\hat{Y}_{k}^{(t, l)}]_{i\last, l} \in (0, \omega)$.
	This indicates that $l \neq b$.
	Moreover, if $[\hat{W}_{k}^{(t, b)}]_{:, l} = 0$, we know that $[\hat{W}_{k}^{(t, l)}]_{i, l} = [\hat{W}_{k}^{(t, b)}]_{i, l} = 0$ for all $i \neq i\last$.
	Then the closed-form expression \eqref{eq:sol-subp-cl} for $v = l$ implies that either $[\hat{Y}_{k}^{(t, l)}]_{i\last, l} = 0$ or $[\hat{Y}_{k}^{(t, l)}]_{i\last, l} = 1$, which also results in a contradiction.
	Hence, we can obtain that $[\hat{W}_{k}^{(t, b)}]_{:, l} \neq 0$.
	And it follows from Proposition \ref{prop:supp} that
	\begin{equation*}
		\bar{f}_{Y_{k}} (\hat{Y}_{k}^{(t, l)}) - \bar{f}_{Y_{k}} (\hat{Y}_{k}^{(t, b)})
		= \norm{ [\hat{W}_{k}^{(t, b)}]_{:, b} }_2
		+ \alpha_{k, b}^{(t, l)}
		+ \norm{ [\hat{W}_{k}^{(t, b)}]_{:, l} }_2
		- \norm{ [\hat{W}_{k}^{(t, l)}]_{:, l} }_2,
	\end{equation*}
	where
	\begin{equation*}
		\alpha_{k, b}^{(t, l)} =
		\left\{
		\begin{aligned}
			& - \norm{ [\hat{W}_{k}^{(t, l)}]_{:, b} }_2,
			&& \mbox{if } [\hat{W}_{k}^{(t, l)}]_{:, b} \neq 0, \\
			& \min \{ [\nabla f (Y_{k}) - \eta Y_{k}]_{i, b} \mid i \in \supp ([\hat{S}_{k}^{(t, l)}]_{:, b}) \},
			&& \mbox{otherwise}.
		\end{aligned}
		\right.
	\end{equation*}
	By invoking the triangle inequality, we arrive at
	\begin{equation*}
		\begin{aligned}
			\abs{ \norm{ [\hat{W}_{k}^{(t, b)}]_{:, l} }_2 - \norm{ [\hat{W}_{k}^{(t, l)}]_{:, l} }_2 }
			\leq {} & \norm{ [\hat{W}_{k}^{(t, b)}]_{:, l} - [\hat{W}_{k}^{(t, l)}]_{:, l} }_2
			= [\hat{W}_{k}^{(t, l)}]_{i\last, l} \\
			= {} & [\hat{Y}_{k}^{(t, l)}]_{i\last, l} \norm{ [\hat{W}_{k}^{(t, l)}]_{:, l} }_2
			< \omega (\eta + M).
		\end{aligned}
	\end{equation*}
	Collecting the above two relationships together yields that
	\begin{equation*}
		\bar{f}_{Y_{k}} (\hat{Y}_{k}^{(t, l)}) - \bar{f}_{Y_{k}} (\hat{Y}_{k}^{(t, b)})
		> \norm{ [\hat{W}_{k}^{(t, b)}]_{:, b} }_2 + \alpha_{k, b}^{(t, l)}  - \omega (\eta + M).
	\end{equation*}
	Now we investigate the following two cases.
	
	\paragraph{Case I:} $[\hat{W}_{k}^{(t, l)}]_{:, b} \neq 0$.
	Then it holds that $\alpha_{k, b}^{(t, l)} = - \| [\hat{W}_{k}^{(t, l)}]_{:, b} \|_2$.
	By simple calculations, we can obtain that
	\begin{equation*}
		\begin{aligned}
			\norm{ [\hat{W}_{k}^{(t, b)}]_{:, b} }_2
			- \norm{ [\hat{W}_{k}^{(t, l)}]_{:, b} }_2
			= {} & \dfrac{ \norm{ [\hat{W}_{k}^{(t, b)}]_{:, b} }_2^2 - \norm{ [\hat{W}_{k}^{(t, l)}]_{:, b} }_2^2 }{ \norm{ [\hat{W}_{k}^{(t, b)}]_{:, b} }_2 + \norm{ [\hat{W}_{k}^{(t, l)}]_{:, b} }_2 } \\
			\geq {} & \dfrac{1}{2 (\eta + M)} \dkh{ \eta [Y_{k}]_{i\last, b} - [\nabla f (Y_{k})]_{i\last, b} }^2
			\geq \dfrac{\tau^2}{8 (\eta + M)}.
		\end{aligned}
	\end{equation*}
	As a result, it can be readily verified that
	\begin{equation*}
		\bar{f}_{Y_{k}} (\hat{Y}_{k}^{(t, l)}) - \bar{f}_{Y_{k}} (\hat{Y}_{k}^{(t, b)})
		> \dfrac{\tau^2}{8 (\eta + M)} - \omega (\eta + M)
		\geq 0,
	\end{equation*}
	which stands in contradiction to the definition of $l = v^{(t)}$.
	
	\paragraph{Case II:} $[\hat{W}_{k}^{(t, l)}]_{:, b} = 0$.
	In this case, we have $\alpha_{k, b}^{(t, l)} \geq 0$.
	A straightforward verification reveals that
	\begin{equation*}
		\begin{aligned}
			\bar{f}_{Y_{k}} (\hat{Y}_{k}^{(t, l)}) - \bar{f}_{Y_{k}} (\hat{Y}_{k}^{(t, b)})
			> {} & [\hat{W}_{k}^{(t, b)}]_{i\last, b} - \omega (\eta + M) \\
			= {} & \eta [Y_{k}]_{i\last, b} - [\nabla f (Y_{k})]_{i\last, b} - \omega (\eta + M)
			\geq \dfrac{\tau}{2} - \omega (\eta + M)
			\geq 0.
		\end{aligned}
	\end{equation*}
	Similarly, the above relationship contradicts the definition of $l = v^{(t)}$.
	
	From the combination of the foregoing two cases, it follows that condition \eqref{eq:sc-zrow} is also satisfied when $\bK$ is finite.
	Therefore, we conclude that the accumulation point $X\last \in \Opnp$ is indeed a first-order stationary point of problem \eqref{opt:stplus}.
	The proof is completed.
\end{proof}

As elucidated in the proof of Theorem \ref{thm:convergence}, whenever a row contains a position with a negative Euclidean gradient, the update scheme for support sets inevitably assigns it a nonzero entry.
This guarantees that every zero row adheres to condition \eqref{eq:sc-zrow}.
Another stationarity condition \eqref{eq:sc-supp}, in turn, is enforced by solving subproblem \eqref{opt:subp-zrow}.
Collectively, these mechanisms ensure that our algorithm converges to a first-order stationary point of problem \eqref{opt:stplus}.

\subsection{Iteration Complexity}

Next, we are in the position to establish the iteration complexity of Algorithm \ref{alg:supps} to find an approximate first-order stationary point.

\begin{theorem}
	\label{thm:complexity}
	For any $\epsilon \in (0, 1)$, Algorithm \ref{alg:supps} will terminate at an $\epsilon$-approximate first-order stationary point of problem \eqref{opt:stplus} after at most $O (\epsilon^{-2})$ iterations.
\end{theorem}

\begin{proof}
	For any $\epsilon \in (0, 1)$, we define
	\begin{equation*}
		k_{\epsilon} = \min \hkh{ k\uast \;\middle|\; k\uast \in \argmin_{ k \in \{0, 1, \dotsc, K_{\epsilon}\} } \norm{Y_k - X_k}\ffs + \norm{X_{k + 1} - Y_k}\ffs },
	\end{equation*}
	where
	\begin{equation}
		\label{eq:kepsilon}
		K_{\epsilon} = \left\lceil \dfrac{2 (\overline{f} - \underline{f})}{\eta - L} \max \hkh{ \dfrac{2}{\theta^2}, \, \dfrac{4 (\eta + L)^2}{\epsilon^2}, \, \dfrac{(\eta + M)^2}{\epsilon^2} } \right\rceil.
	\end{equation}
	Now it follows from the relationship \eqref{eq:sum-norm} that
	\begin{equation*}
		\begin{aligned}
			\norm{ Y_{k_{\epsilon}} - X_{k_{\epsilon}} }\ffs + \norm{ X_{k_{\epsilon} + 1} - Y_{k_{\epsilon}} }\ffs
			\leq {} & \dfrac{1}{K_{\epsilon}} \sum_{k = 0}^{K_{\epsilon} - 1} \dkh{ \norm{Y_k - X_k}\ffs + \norm{X_{k + 1} - Y_k}\ffs } \\
			\leq {} & \dfrac{2 (\overline{f} - \underline{f})}{(\eta - L) K_{\epsilon}}
			\leq \min \hkh{ \dfrac{\theta^2}{2}, \, \dfrac{\epsilon^2}{4 (\eta + L)^2}, \, \dfrac{\epsilon^2}{(\eta + M)^2} }.
		\end{aligned}
	\end{equation*}
	As a direct consequence of Proposition \ref{prop:fosc}, we can proceed to show that
	\begin{equation*}
		\abs{[\grad\, f (Y_{k_{\epsilon}})]_{i, j}}
		\leq 2 (\eta + L) \norm{ Y_{k_{\epsilon}} - X_{k_{\epsilon}} }\ff
		\leq \epsilon,
	\end{equation*}
	for all $(i, j) \in \supp (Y_{k_{\epsilon}})$.
	Moreover, since $\| Y_{k_{\epsilon}} - X_{k_{\epsilon}} \|\ff < \theta$, it holds that
	\begin{equation*}
		\max\{0, - [\nabla f (Y_{k_{\epsilon}})]_{i, j}\}
		\leq (\eta + M) \norm{ X_{k_{\epsilon} + 1} - Y_{k_{\epsilon}} }\ff
		\leq \epsilon,
	\end{equation*}
	for all $i \in \zrow (Y_{k_{\epsilon}})$ and $j \in \{1, 2, \dotsc, p\}$.
	Therefore, we conclude that $Y_{k_{\epsilon}} \in \Opnp$ is an $\epsilon$-approximate first-order stationary point of problem \eqref{opt:stplus}, which can be obtained by Algorithm \ref{alg:supps} after at most $K_{\epsilon} = O (\epsilon^{-2})$ iterations.
	The proof is completed.
\end{proof}

Theorem \ref{thm:complexity} clarifies that the iteration complexity of Algorithm~\ref{alg:supps} is $K_{\epsilon} = O (\epsilon^{-2})$ to attain an $\epsilon$-approximate first-order stationary point.
This iteration complexity, to the best of our knowledge, represents the first such result for optimization problems with nonnegative and orthogonal constraints in the literature.
As revealed by the expression of $K_{\epsilon}$ in \eqref{eq:kepsilon}, the parameter $\theta$ enters only through a term independent of $\epsilon$. 
Consequently, for sufficiently small $\epsilon$, the iteration complexity of Algorithm~\ref{alg:supps} is unaffected by the choice of $\theta$.

\begin{remark}
	Theorem~\ref{thm:convergence} asserts that every accumulation point of $\{X_k\}$ is a first-order stationary point, whereas Theorem~\ref{thm:complexity} identifies the intermediate iterate $Y_{k_{\epsilon}}$, rather than $X_{k_{\epsilon}}$, as an $\epsilon$-approximate first-order stationary point. 
	These two conclusions are consistent. 
	Indeed, Corollary~\ref{coro:dist-iter} shows that $\norm{Y_k - X_k}\ff \to 0$ and $\norm{X_{k + 1} - Y_k}\ff \to 0$ as $k \to \infty$, thereby ensuring that $\{X_k\}$ and $\{Y_k\}$ share the same accumulation points. 
	Furthermore, Proposition~\ref{prop:fosc} quantifies the stationarity violation at $Y_k$ in terms of $\norm{Y_k - X_k}\ff$ and $\norm{X_{k + 1} - Y_k}\ff$. 
	Therefore, the proof of Theorem~\ref{thm:complexity} naturally selects $Y_{k_{\epsilon}}$ as the candidate for an $\epsilon$-approximate first-order stationary point. 
\end{remark}

\subsection{Finite Support Identification}

This subsection demonstrates that the support set of stationary points can be identified after a finite number of iterations.

\begin{theorem}
	\label{thm:support}
	Let $\{(X_k, Y_k)\}$ be the sequence generated by Algorithm \ref{alg:supps}.
	Then at least one of the following statements holds for each row $i \in \{1, 2, \dotsc, n\}$.
	\begin{enumerate}[(i)]
		
		\item There exists $\Bbbk_i \in \bN$ such that the position of the nonzero entry in $[X_k]_{i, :}$ remains unchanged for all $k \geq \Bbbk_i$.
		
		\item $\liminf_{k \to \infty} \norm{[X_k]_{i, :}}_2 = 0$.
		
	\end{enumerate}
\end{theorem}

\begin{proof}
	We assume that statement $(i)$ does not hold for the $i$-th row.
	Then there exists a sequence $\{k_q\}$ satisfying $\lim_{q \to \infty} k_q = \infty$ such that the nonzero entries in $[X_{k_q}]_{i, :}$ and $[X_{k_q + 1}]_{i, :}$ occur at different positions.
	From the construction of Algorithm~\ref{alg:supps}, it follows that $[X_{k_q}]_{i, :}$ and $[Y_{k_q}]_{i, :}$ share the same position for the nonzero entry.
	Since both $[Y_{k_q}]_{i, :}$ and $[X_{k_q + 1}]_{i, :}$ contain at most one nonzero entry, we have
	\begin{equation*}
		\norm{[Y_{k_q}]_{i, :} - [X_{k_q + 1}]_{i, :}}_2
		\geq \norm{[Y_{k_q}]_{i, :}}_2,
	\end{equation*}
	which together with the triangle inequality implies that
	\begin{equation*}
		\begin{aligned}
			\norm{[X_{k_q}]_{i, :}}_2
			\leq {} & \norm{[X_{k_q}]_{i, :} - [Y_{k_q}]_{i, :}}_2 + \norm{[Y_{k_q}]_{i, :}}_2 \\
			\leq {} & \norm{[X_{k_q}]_{i, :} - [Y_{k_q}]_{i, :}}_2 + \norm{[Y_{k_q}]_{i, :} - [X_{k_q + 1}]_{i, :}}_2.
		\end{aligned}
	\end{equation*}
	According to~Corollary \ref{coro:dist-iter}, we can obtain that $\liminf_{k \to \infty} \norm{[X_k]_{i, :}}_2 = 0$ by passing to the limit $q \to \infty$ in the above relationship.
	Therefore, statement $(ii)$ always holds in this case.
	The proof is completed.
\end{proof}

%The following results arise as two immediate consequences of Theorem~\ref{thm:support}. 
%Their proofs are omitted for brevity. 
The following result arises as an immediate consequence of Theorem~\ref{thm:support}, whose proof is omitted for brevity.

\begin{corollary}
	\label{coro:support}
	Let $\{(X_k, Y_k)\}$ be the sequence generated by Algorithm \ref{alg:supps}. 
	Suppose that every first-order stationary point of problem \eqref{opt:stplus} is free of zero rows. 
	Then $\supp (X_k)$ becomes fixed after a finite number of iterations. 
	%	\begin{enumerate}[(i)]
		%		
		%		\item Suppose that every first-order stationary point of problem \eqref{opt:stplus} is free of zero rows. 
		%		Then $\supp (X_k)$ is fixed after a finite number of iterations. 
		%		
		%		\item For any $\sigma \in (0, 1]$, $\sigma$-$\supp (X_k)$ becomes invariant after a finite number of iterations, where $\sigma$-$\supp (X) := \{ (i, j) \mid \abs{ [X]_{i, j} } > \sigma \}$. 
		%		
		%	\end{enumerate}
\end{corollary}

We present two representative examples to illustrate that the circumstances described in Corollary~\ref{coro:support} are commonly encountered in practice. 
\begin{enumerate}[(i)]
	
	\item The first example is the linear function $f (X) = \tr (A\zz X)$, where each row of the matrix $A \in \Rnp$ contains at least one strictly negative entry.
	
	\item The second example is the quadratic function $f (X) = \tr (X\zz A X)$, where all entries of the symmetric matrix $A \in \Rnn$ are strictly negative.
	
\end{enumerate}
In both cases, it is straightforward to verify that, for any $X \in \Opnp$, each row of $\nabla f (X)$ necessarily contains at least one strictly negative entry.
Consequently, the stationarity condition \eqref{eq:sc-zrow} can never be satisfied in such settings, which in turn ensures that every first-order stationary point must be devoid of zero rows.
%of problem \eqref{opt:stplus}

\subsection{Whole-Sequence Convergence Under the K{\L} Property}

Finally, we close this section by investigating the convergence behavior of Algorithm~\ref{alg:supps} under an additional K{\L} property. 
It is shown that the whole sequence generated by Algorithm~\ref{alg:supps} converges to a single first-order stationary point of problem~\eqref{opt:stplus}.

Let $\zeta_{\Opnp}: \Rnp \to (-\infty, +\infty]$ be the indicator function of $\Opnp$ and $\psi: \Rnp \to (-\infty, +\infty]$ be the extended objective function defined by
\begin{equation*}
	\psi (X) = f (X) + \zeta_{\Opnp} (X).
\end{equation*}
We first provide a normal-cone estimate that connects the limiting subdifferential of $\psi$, denoted by $\partial \psi$, with two stationarity residuals in Proposition~\ref{prop:kkt}.

\begin{lemma}
	\label{le:subdiff-bound}
	For any $Y \in \Opnp$, it holds that
	\begin{equation}
		\label{eq:subdiff-bound}
		\dist^2 (0, \partial \psi (Y))
		\leq \sum_{(i, j) \in \supp (Y)} \abs{[\grad\, f (Y)]_{i, j}}^2
		+ \sum_{i \in \zrow (Y)} \sum_{j = 1}^{p} \max \{0, - [\nabla f (Y)]_{i, j}\}^2.
	\end{equation}
\end{lemma}

\begin{proof}
	Let $D \in \Rnp$ be a matrix defined by
	\begin{equation*}
		[D]_{i, j} =
		\left\{
		\begin{aligned}
			& - \lambda_j [Y]_{i, j}, 
			&& \mbox{if } (i, j) \in \supp (Y), \\
			& - \max \{0, [\nabla f (Y)]_{i, j}\}, 
			&& \mbox{if } i \in \zrow (Y) \mbox{ and } j \in \{1, 2, \dotsc, p\}, \\
			& - [\nabla f (Y)]_{i, j}, 
			&& \mbox{otherwise},
		\end{aligned}
		\right.
	\end{equation*}
	where $\lambda_j = [Y\zz \nabla f (Y)]_{j, j}$ for $j \in \{1, 2, \dotsc, p\}$. 
	From the definition of $\cN_{\Opnp} (Y)$ given in \eqref{eq:normal}, it can be readily verified that $D \in \cN_{\Opnp} (Y)$. 
	Since $f$ is continuously differentiable, we have $\partial \psi (Y) = \nabla f (Y) + \cN_{\Opnp} (Y)$. 
	The construction of $D$ yields that
	\begin{equation*}
		[\nabla f (Y) + D]_{i, j} =
		\left\{
		\begin{aligned}
			& [\grad\, f (Y)]_{i, j}, 
			&& \mbox{if } (i, j) \in \supp (Y), \\
			& - \max \{0, - [\nabla f (Y)]_{i, j}\}, 
			&& \mbox{if } i \in \zrow (Y) \mbox{ and } j \in \{1, 2, \dotsc, p\}, \\
			& 0, 
			&& \mbox{otherwise}.
		\end{aligned}
		\right.
	\end{equation*}
	Then the estimate \eqref{eq:subdiff-bound} follows immediately, due to the fact that $\nabla f (Y) + D \in \partial \psi (Y)$. 
	We complete the proof. 
\end{proof}

With the preceding preparations in place, we are ready to establish the convergence of the entire sequence generated by Algorithm~\ref{alg:supps} when $f$ is a semi-algebraic function.

\begin{theorem}
	\label{thm:kl}
	Suppose that $f$ is semi-algebraic. 
	Then the whole sequence $\{X_k\}$ generated by Algorithm~\ref{alg:supps} converges to a first-order stationary point of problem~\eqref{opt:stplus}. 
\end{theorem}

\begin{proof}
	The feasible set $\Opnp$ is described by polynomial equalities and inequalities, which is semi-algebraic. 
	Since $f$ is semi-algebraic, the extend objective function $\psi$ is proper, lower semi-continuous, and semi-algebraic.
	Consequently, it follows from \cite[Theorem~3]{Bolte2014proximal} that $\psi$ possesses the K{\L} property. 
	
	We shall prove the assertion by establishing the finite-length property of the intermediate sequence $\{Y_k\}$. 
	The two descent relationships \eqref{eq:des-f-xk} and \eqref{eq:des-f-yk} imply that
	\begin{equation}
		\label{eq:des-yk-kl}
		\begin{aligned}
			\psi (Y_k) - \psi (Y_{k + 1})
			= {} & f (Y_k) - f (Y_{k + 1})
			= f (Y_k) - f (X_{k + 1}) + f (X_{k + 1}) - f (Y_{k + 1}) \\
			\geq {} & c_1 \norm{X_{k + 1} - Y_k}\ffs + c_1 \norm{Y_{k + 1} - X_{k + 1}}\ffs,
		\end{aligned}
	\end{equation}
	where $c_1 = (\eta - L) / 2 > 0$ is a constant. 
	We obtain that the sequence $\{\psi (Y_k)\}$ is monotonically nonincreasing. 
	The relationship \eqref{eq:bound-f} guarantees that $\{\psi (Y_k)\}$ is bounded from below and hence converges to some $f\last \in \bR$. 
	If there exists $\bar{k}_0$ such that $\psi (Y_{\bar{k}_0}) = f\last$, we can deduce from \eqref{eq:des-yk-kl} that $Y_{\bar{k}_0 + 1} = X_{\bar{k}_0 + 1} = Y_{\bar{k}_0}$. 
	Then a trivial induction shows that the assertion of this theorem is obvious. 
	Hence, we focus on the case where $\psi (Y_k) > f\last$ for all $k \in \bN$. 
	For any $\mu > 0$, there exists $\bar{k}_1 \in \bN$ such that $f\last < \psi (Y_k) < f\last + \mu$ with $k \geq \bar{k}_1$. 
	Let $\Omega$ denote the set containing all accumulation points of $\{Y_k\}$. 
	According to \cite[Lemma~5]{Bolte2014proximal}, the set $\Omega$ is nonempty and compact, $\psi$ is constantly equal to $f\last$ on $\Omega$, and $\dist (Y_k, \Omega) \to 0$ as $k \to \infty$. 
	Thus, for any $\kappa > 0$, there exists $\bar{k}_2 \in \bN$ such that $\dist (Y_k, \Omega) \leq \kappa$ with $k \geq \bar{k}_2$. 
	Summing up all these facts, we obtain that $Y_k$ belongs to the intersection of $\{Y \in \Opnp \mid f\last < \psi (Y) < f\last + \mu\}$ and $\{Y \in \Opnp \mid \dist (Y, \Omega) \leq \kappa\}$ for all $k \geq \max\{\bar{k}_1, \bar{k}_2\}$.

	By invoking the uniformized K{\L} property in \cite[Lemma~6]{Bolte2014proximal}, there exists a continuous and concave function $\chi \in \Phi_\mu$ such that
	\begin{equation}
		\label{eq:uniform-kl}
		\chi^{\prime} (\psi (Y_k) - f\last) \, \dist (0, \partial \psi (Y_k)) \geq 1,
	\end{equation}
	for all $k \geq \max\{\bar{k}_1, \bar{k}_2\}$. 
	The concavity of $\chi$, together with two inequalities \eqref{eq:des-yk-kl} and \eqref{eq:uniform-kl}, yields that
	\begin{equation}
		\label{eq:chi-dist}
		\begin{aligned}
			\chi (\psi (Y_k) - f\last) - \chi (\psi (Y_{k + 1}) - f\last)
			\geq {} & \chi^{\prime} (\psi (Y_k) - f\last) (\psi (Y_k) - \psi (Y_{k + 1})) \\
			\geq {} & \dfrac{c_1 \norm{X_{k + 1} - Y_k}\ffs + c_1 \norm{Y_{k + 1} - X_{k + 1}}\ffs}{\dist (0, \partial \psi (Y_k))}.
		\end{aligned}
	\end{equation}
	By Corollary~\ref{coro:dist-iter}, there exists $\bar{k}_3 \in \bN$ such that $\norm{Y_k - X_k}\ff < \theta$ for all $k \geq \bar{k}_3$. 
	Let $c_2 = 4 n (\eta + L)^2 > 0$ and $c_3 = p (n - p) (\eta + M)^2 > 0$ be two constants. 
	Then Lemma~\ref{le:subdiff-bound} and Proposition~\ref{prop:fosc} together imply that
	\begin{equation}
		\label{eq:rel-error-yk}
		\begin{aligned}
			\dist^2 (0, \partial \psi (Y_k))
			\leq {} & \sum_{(i, j) \in \supp (Y_k)} \abs{[\grad\, f (Y_k)]_{i, j}}^2
			+ \sum_{i \in \zrow (Y_k)} \sum_{j = 1}^{p} \max \{0, - [\nabla f (Y_k)]_{i, j}\}^2 \\
			\leq {} & c_2 \norm{Y_k - X_k}\ffs
			+ c_3 \norm{X_{k + 1} - Y_k}\ffs,
		\end{aligned}
	\end{equation}
	for all $k \geq \bar{k}_3$. 
	Combining \eqref{eq:chi-dist} with \eqref{eq:rel-error-yk}, we immediately arrive at
	\begin{equation}
		\label{eq:chi-rk}
		\chi (\psi (Y_k) - f\last) - \chi (\psi (Y_{k + 1}) - f\last)
		\geq \dfrac{c_1 \norm{X_{k + 1} - Y_k}\ffs + c_1 \norm{Y_{k + 1} - X_{k + 1}}\ffs}{\sqrt{c_2 \norm{Y_k - X_k}\ffs + c_3 \norm{X_{k + 1} - Y_k}\ffs}},
	\end{equation}
	for all $k \geq \bar{k} := \max\{\bar{k}_1, \bar{k}_2, \bar{k}_3\}$. 
	For notational simplicity, we denote
	\begin{equation*}
		s_k = \chi (\psi (Y_k) - f\last) - \chi (\psi (Y_{k + 1}) - f\last),
		\quad\mbox{and}\quad
		d_k = \norm{X_{k + 1} - Y_k}\ff + \norm{Y_{k + 1} - X_{k + 1}}\ff. 
	\end{equation*}
	It can be readily verified that
	\begin{equation}
		\label{eq:dk}
		d_k^2 \leq 2 \norm{X_{k + 1} - Y_k}\ffs + 2 \norm{Y_{k + 1} - X_{k + 1}}\ffs.
	\end{equation}
	As $d_k = \norm{X_{k + 1} - Y_k}\ff + \norm{Y_{k + 1} - X_{k + 1}}\ff \geq \norm{X_{k + 1} - Y_k}\ff$ and $d_{k - 1} = \norm{X_k - Y_{k - 1}}\ff + \norm{Y_k - X_k}\ff \geq \norm{Y_k - X_k}\ff$, we have
	\begin{equation}
		\label{eq:dk-1}
		\dkh{ d_k + d_{k - 1} }^2
		\geq \dkh{ \norm{X_{k + 1} - Y_k}\ff + \norm{Y_k - X_k}\ff }^2 
		\geq \norm{X_{k + 1} - Y_k}\ffs + \norm{Y_k - X_k}\ffs.
	\end{equation}
	Collecting three relationships \eqref{eq:chi-rk}, \eqref{eq:dk}, and \eqref{eq:dk-1} together directly yields that
	\begin{equation}
		\label{eq:rk-recursive}
		d_k^2 \leq c_4 s_k \dkh{d_k + d_{k - 1}}, 
	\end{equation}
	where $c_4 = 2 \sqrt{\max \{c_2, c_3\}} / c_1 > 0$ is a constant.

	A straightforward verification reveals that the elementary inequality $\sqrt{a b} \leq a / 4 + b$ holds for $a, b \geq 0$. 
	Then the relationship \eqref{eq:rk-recursive} further gives
	\begin{equation}
		\label{eq:rk-linear}
		d_k \leq \dfrac{1}{3} d_{k - 1} + \dfrac{4}{3} c_4 s_k.
	\end{equation}
	Summing \eqref{eq:rk-linear} over $k$ from $\bar{k}$ to $K$, we have
	\begin{equation*}
		\sum_{k = \bar{k}}^{K} d_k
		\leq \dfrac{1}{3} d_{\bar{k} - 1} + \dfrac{1}{3} \sum_{k = \bar{k}}^{K} d_k
		+ \dfrac{4}{3} c_4 \sum_{k = \bar{k}}^{K} s_k.
	\end{equation*}
	Since $\sum_{k = \bar{k}}^{K} s_k = \chi (\psi (Y_{\bar{k}}) - f\last) - \chi (\psi (Y_{K + 1}) - f\last) \leq \chi (\psi (Y_{\bar{k}}) - f\last)$, it follows that
	\begin{equation*}
		\sum_{k = \bar{k}}^{K} d_k \leq \dfrac{1}{2} d_{\bar{k} - 1} + 2 c_4 \chi (\psi (Y_{\bar{k}}) - f\last),
	\end{equation*}
	which, by passing to the limit $K \to \infty$, implies that
	\begin{equation*}
		\sum_{k = \bar{k}}^{\infty} \norm{Y_{k + 1} - Y_k}\ff
		\leq \sum_{k = \bar{k}}^{\infty} \dkh{ \norm{Y_{k + 1} - X_{k + 1}}\ff + \norm{X_{k + 1} - Y_k}\ff } 
		= \sum_{k = \bar{k}}^{\infty} d_k
		< \infty.
	\end{equation*}
	Therefore, $\{Y_k\}$ is a Cauchy sequence and converges to a point $X\last \in \Opnp$. 
	In view of Corollary~\ref{coro:dist-iter}, we also conclude that $\{X_k\}$ converges to $X\last$. 
	Finally, Theorem~\ref{thm:convergence} ensures that $X\last$ is a first-order stationary point of problem~\eqref{opt:stplus}. 
	The proof is completed. 
\end{proof}

Leveraging the K{\L} property, Theorem~\ref{thm:kl} strengthens the theoretical guarantee of Algorithm~\ref{alg:supps} by establishing the convergence of the entire sequence.

\section{Numerical Results}

\label{sec:numerical}

Preliminary numerical results are presented in this section to validate the effectiveness and efficiency of the proposed algorithm. 
All algorithms are implemented in MATLAB R2018b on a workstation with dual Intel Xeon Gold 6242R CPU processors (at $3.10$ GHz$\times 20 \times 2$) and $510$ GB of RAM under Ubuntu 20.04. 
Our codes are publicly available online\footnote{See \url{https://github.com/LeiWang-Opt/Support-Set}.}.

\subsection{Implementation Details}

\label{subsec:implementation}

In our algorithm, solving subproblem~\eqref{opt:subp-supp} can be interpreted as performing a projected gradient step with $1 / \eta$ serving as the associated stepsize, in which the point $Z - \nabla f (Z) / \eta$ is projected onto the set $\{X \in \Opnp \mid \supp (X) \subseteq \supp (S)\}$.
Drawing inspiration from \cite{Wang2022decentralized,Wen2013feasible}, we employ the Barzilai-Borwein (BB) stepsize scheme \cite{Barzilai1988two} to update $\eta$.
Specifically, at each iteration $k$, the parameter $\eta$ in both subproblem \eqref{opt:subp-zrow} and subproblem \eqref{opt:subp-cl} is adaptively selected as
\begin{equation*}
	\eta_k = \dfrac{\abs{\jkh{X_{k} - X_{k - 1}, \nabla f (X_{k}) - \nabla f (X_{k - 1})}}}{\norm{X_{k} - X_{k - 1}}\ffs}.
\end{equation*}
Our empirical findings suggest that adopting this stepsize scheme leads to a noticeable acceleration of the convergence rate in practice.
Similar strategies are likewise utilized in \cite{Jiang2023exact,Qian2024error}.

Two parameters, $\theta$ and $\delta$, are of particular importance in Algorithm~\ref{alg:supps}, which determines when the support set is switched and controls the number of nonzero entries adjusted during each transition, respectively. 
To assess their influence, we evaluate the behavior of Algorithm 1 under various parameter choices on the following nonnegative PCA \cite{Montanari2015non} problem,
\begin{equation}
	\label{opt:npca}
	\min_{X \in \Opnp} \hspace{2mm} - \dfrac{1}{2} \tr \dkh{X\zz A\zz A X},
\end{equation}
where $A \in \Rmn$ is a data matrix with $m \leq n$. 
This experiment sets $n = 5000$, $m = 500$, and $p = 50$. 
We randomly generate the matrix $A$ with its entries sampled from a normal distribution. 
The initial point is fixed as the matrix formed by the first $p$ columns of $I_n$. 
Algorithm~\ref{alg:supps} is terminated when $\norm{X_{k + 1} - X_{k}}\ff \leq 10^{-6}$. 
Figure~\ref{fig:para} illustrates, in its two subplots, the decrease of function values over CPU time for $\theta \in \{10^{-1}, 10^{-2}, 10^{-3}\}$ and $\delta \in \{0.05, 0.1, 0.15\}$, respectively. 
As evidenced therein, Algorithm~\ref{alg:supps} demonstrates a remarkable degree of stability across different parameter choices, with its performance exhibiting no pronounced sensitivity to the specific values selected. 
Accordingly, Algorithm~\ref{alg:supps} is configured with $\theta = 10^{-2}$ and $\delta = 0.1$ by default.

\begin{figure}[t]
	\centering
	\subfigure[Different $\theta$]{
		\label{subfig:para_theta}
		\includegraphics[width=0.32\linewidth]{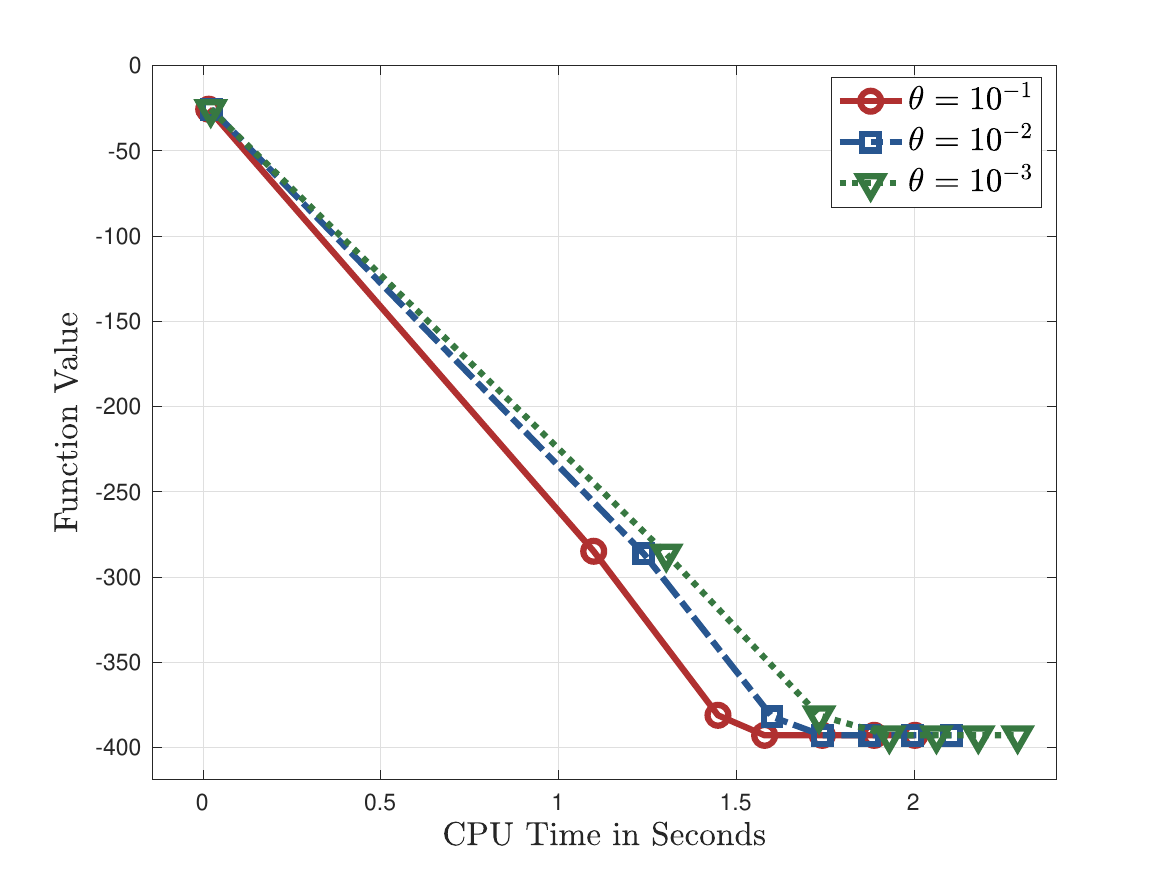}
	}
	\subfigure[Different $\delta$]{
		\label{subfig:para_delta}
		\includegraphics[width=0.32\linewidth]{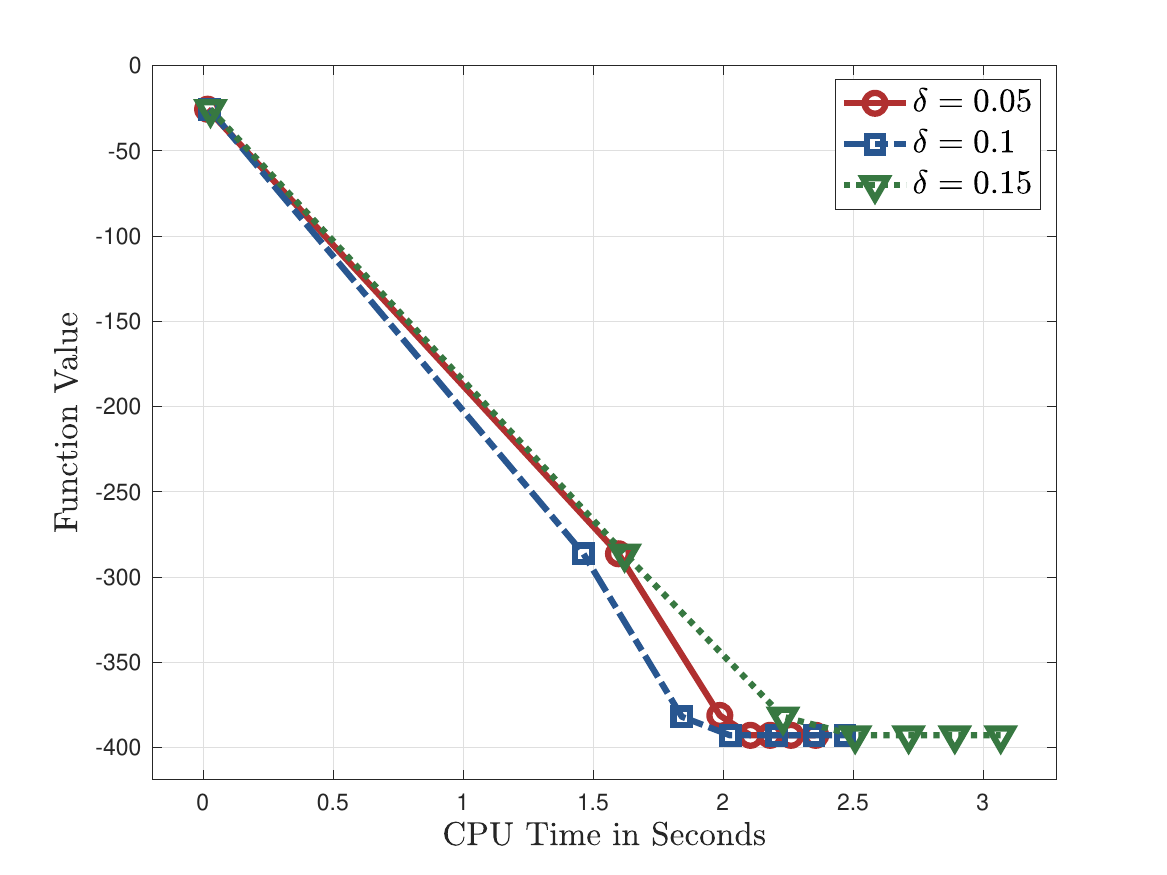}
	}
	\caption{Numerical comparison of different algorithmic parameters for \SUPPORT.}
	\label{fig:para}
\end{figure}

In the subsequent subsections, we conduct a comprehensive performance comparison between the proposed algorithm \SUPPORT and four existing methods---\EPOrth~\cite{Jiang2023exact}, \SEPPGz~\cite{Qian2024error}, \SEPPGp~\cite{Qian2024error}, and \OptM~\cite{Wen2013feasible}---on a variety of testing problems. 
The first three benchmark algorithms have been introduced in Section~\ref{subsec:prior}. 
For our experiments, we select the augmented Lagrangian method available in \OptM, which is described thoroughly in \cite[Section~5.2]{Qian2024error}. 
The implementations of these methods are sourced from GitHub\footnote{See \url{https://github.com/mengxianglgal/Ep4orth} for \EPOrth, \url{https://github.com/styluck/dSNCG} for \SEPPGz and \SEPPGp, and \url{https://github.com/optsuite/OptM} for \OptM.}.
We retain the original settings and configure the parameters in accordance with the specifications provided in \cite{Jiang2023exact,Qian2024error,Wen2013feasible}.
\EPOrth chooses $V = \mathbf{1}_p / \sqrt{p}$ in \eqref{eq:oblique}, where $\mathbf{1}_p$ is the $p$-dimensional vector of all ones. 
The stopping criterion for \EPOrth is set as $| 1 - \| X_k V \|\ffs | \leq 10^{-6}$. 
The remaining three benchmark algorithms adopt $\norm{ \max \{0, - X_k\} }_1 \leq 10^{-6}$ as the stopping criterion. 
Moreover, all algorithms are terminated once the iteration count reaches the prescribed maximum of $1000$.
Given that these algorithms are infeasible by design, they are generally incapable of producing a feasible solution within a finite number of iterations, which fundamentally distinguishes them from \SUPPORT. 
For a fair and meaningful comparison, the final iterates they return are rounded by \cite[Procedure 1]{Jiang2023exact} to generate feasible solutions.

\subsection{Nonnegative PCA}

We first engage in a numerical comparison of different algorithms on the nonnegative PCA problem \eqref{opt:npca} with the data matrix formulated by the singular value decomposition $A = U \Sigma V\zz$. 
Here, $U \in \cO^{m, m}$ is an orthonormalization of a randomly generated matrix, and $\Sigma \in \bR^{m \times m}$ is a diagonal matrix with randomly sampled positive entries arranged in descending order along the diagonal.
The orthogonal matrix $V \in \cO^{n, m}$ is constructed through a more elaborate procedure.
Specifically, we begin by generating a nonnegative and orthogonal matrix $X_{\opt} \in \Opnp$, which can be achieved by randomly selecting its support set and normalizing each column to have unit norm.
%Specifically, we first generate a nonnegative and orthogonal matrix $X_{opt} \in \Opnp$ at random, in accordance with the strategy outlined in the preceding subsection.
An orthogonal completion $\bar{V} \in \cO^{n, m - p}$ is then computed such that the concatenated matrix $V = [X_{\opt} \; \bar{V}]$ remains orthogonal.
By this design, $X_{\opt} \in \Opnp$ constitutes a global minimizer of problem~\eqref{opt:npca}, as its columns align with the eigenvectors of $A\zz A$ associated with the largest $p$ eigenvalues.
Consequently, the optimal value $f_{\opt}$ of problem \eqref{opt:npca} can be determined from $X_{\opt}$.

Three performance metrics are collected and recorded in our experiments.
The first one is the distance between the point $X_{\alg}$ returned by the algorithm and the global minimizer $X_{\opt}$.
It is noteworthy that, the global minimizer of problem \eqref{opt:npca} is not unique, since $X_{\opt} Q$ also qualifies as a global minimizer if $Q \in \Opp$ and $X_{\opt} Q \in \Opnp$.
Indeed, every permutation matrix naturally complies with these two requirements.
To mitigate this effect, we use the subspace distance \cite{Wang2023communication,Wang2024seeking} to measure the discrepancy as $\dist (X_{\alg}, X_{\opt}) = \| X_{\alg} X_{\alg}\zz - X_{\opt} X_{\opt}\zz \|\ff$.
The second one is the relative gap $(f_{\alg} - f_{\opt}) / (1 + \abs{f_{\opt}})$ between the final function value $f_{\alg}$ achieved by the algorithm and the optimal value $f_{\opt}$.
The third one is the CPU time required by the algorithm.

For our testing, we fix $n = 1000$ and $m = 600$ in problem \eqref{opt:npca}, while varying $p$ across the values in $\{100, 200, 300, 400, 500, 600\}$. 
All algorithms start from the same initial point $X_{\init} \in \Opnp$, which is randomly generated using the same procedure as for $X_{\opt}$.
%We observe that all the algorithms return points that can be considered feasible, as their feasibility violations are on the order of $10^{-16}$.
Figure~\ref{fig:npca} depicts the numerical performances of the tested algorithms, with the three subplots corresponding to the three metrics described earlier.
It is evident that existing algorithms fail to identify the global minimizer of problem \eqref{opt:npca} in some test instances, which is also corroborated by the relative gaps of the achieved function values.
In terms of CPU time, \SUPPORT significantly outperforms the other four existing algorithms.
When $p = 600$, \OptM takes around $200$ seconds while \SUPPORT requires only about $7$ second, resulting in a speedup of nearly $30$ times.
This computational superiority will become even more pronounced if the value of $p$ increases further.

\begin{figure}[t]
	\centering
	\subfigure[CPU Time in Seconds]{
		\label{subfig:npca_time}
		\includegraphics[width=0.3\linewidth]{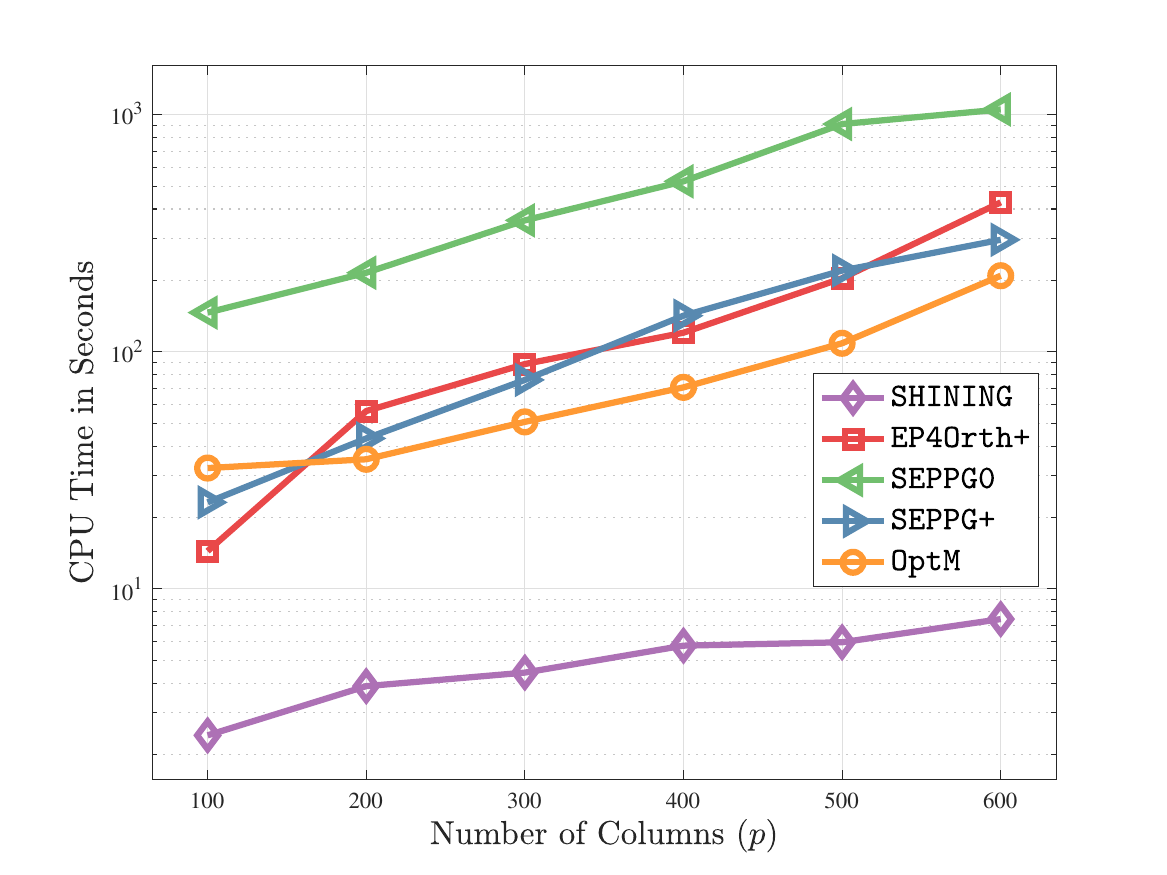}
	}
	\subfigure[Distance from Global Minimizers]{
		\label{subfig:npca_dist}
		\includegraphics[width=0.3\linewidth]{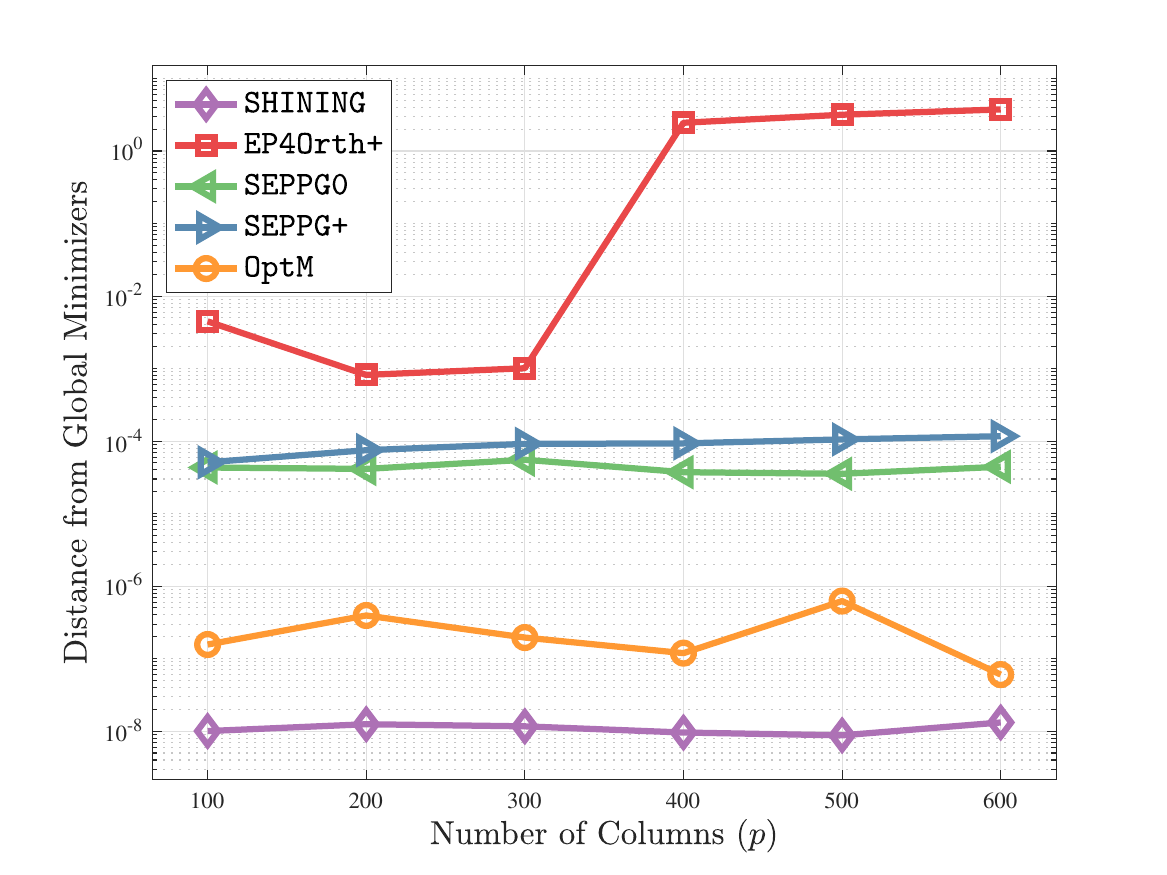}
	}
	\subfigure[Relative Gap of Function Values]{
		\label{subfig:npca_fval}
		\includegraphics[width=0.3\linewidth]{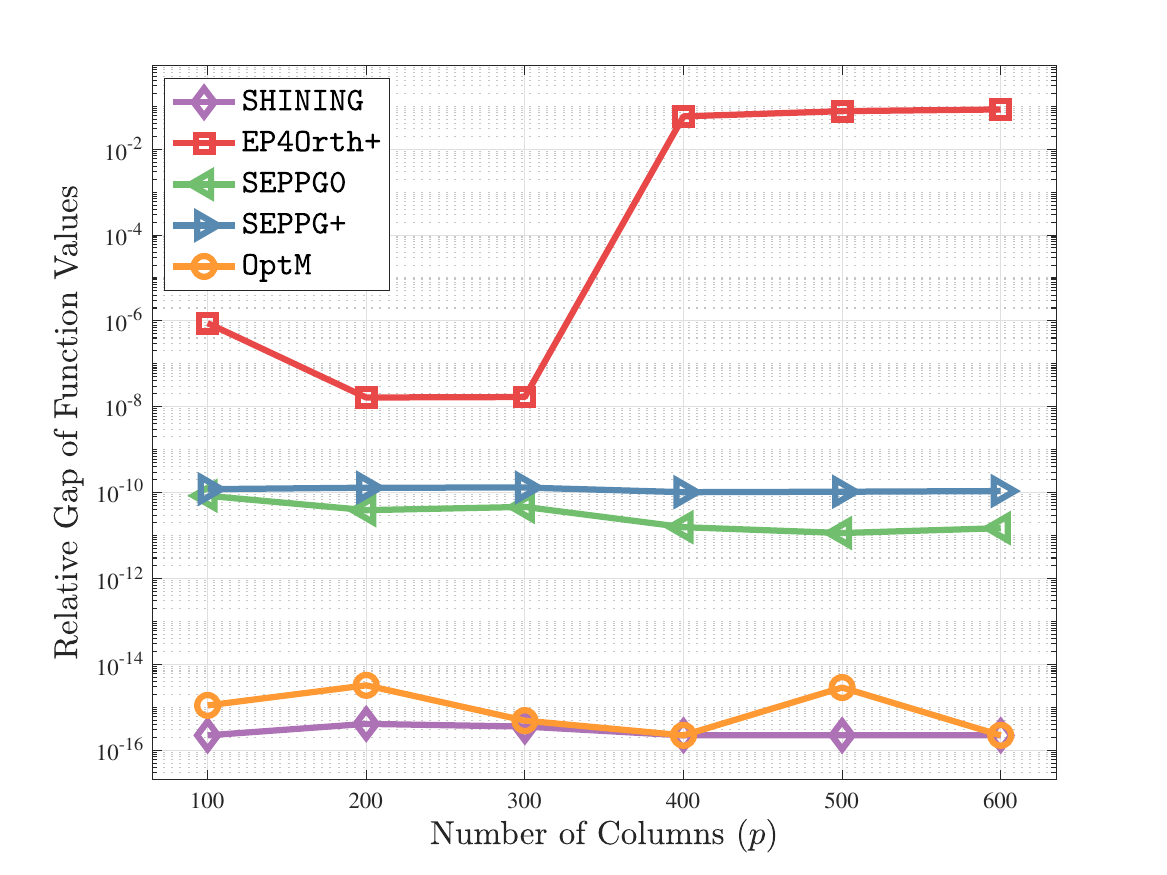}
	}
	\caption{Numerical comparison of different algorithms for solving nonnegative PCA problems.}
	\label{fig:npca}
\end{figure}

\subsection{Image and Text Clustering}

The next series of experiments performs the clustering analysis over real-world datasets by solving the orthogonal nonnegative matrix factorization \cite{Yang2010linear} model formulated as follows,
\begin{equation}
	\label{opt:onmf}
	\min_{X \in \Opnp} \hspace{2mm} \dfrac{1}{2} \norm{A - X X\zz A}\ffs,
\end{equation}
where $A \in \Rnm$ is a data matrix.
The purpose of this problem is to partition $n$ data points, each represented as an $m$-dimensional vector, into $p$ clusters.
We evaluate the performance of the tested algorithms on a collection of  image and text datasets adopted from \cite{Cai2008modeling}, which are publicly available online\footnote{See \url{http://www.cad.zju.edu.cn/home/dengcai/Data/data.html}.}.
For image datasets, each data point is a vector capturing the grayscale values of pixels in a picture.
While for text datasets, every document is encoded as a vector, which reflects the frequency of each word in an article.
The details of the datasets used in this study are summarized in Table \ref{tb:onmf}, which are preprocessed following the same procedure described in \cite[Section 5.2.1]{Jiang2023exact}.
We generate the initial point for all algorithms based on the eigenvectors of $A A\zz$ associated with the largest $p$ eigenvalues.

\begin{table}[t]
	\centering
	\begin{minipage}{\textwidth}
		\caption{Description of datasets for clustering (D1: Yale, D2: TDT2-l10, D3: TDT2-l20, D4: TDT2-t10, D5: TDT2-t20, D6: Reuters-t10, D7: Reuters-t20, D8: NewsG-t5).}
		\label{tb:onmf}
		\begin{tabular*}{\textwidth}{@{\hspace{5pt}\extracolsep{\fill}}ccccccccc@{\hspace{5pt}}}
			\toprule
			Dataset & D1 & D2 & D3 & D4 & D5 & D6 & D7 & D8 \\
			\midrule
			$n$ & $165$ & $653$ & $1938$ & $1477$ & $1721$ & $1897$ & $2402$ & $2344$ \\
			$m$ & $1024$ & $13684$ & $20845$ & $22181$ & $23674$ & $12444$ & $13568$ & $14475$ \\
			$p$ & $15$ & $10$ & $20$ & $10$ & $20$ & $10$ & $20$ & $5$ \\
			\bottomrule
		\end{tabular*}
	\end{minipage}
\end{table}

Any point $X \in \Opnp$ indicates a clustering assignment of a dataset.
To assess the quality of clustering results, we adopt three widely used criteria: entropy \cite{Zhao2004empirical}, purity \cite{Ding2006orthogonal}, and normalized mutual information (NMI) \cite{Xu2003document}.
%The formal definitions of these criteria are omitted for brevity, which can be found in \cite[Section 5.6]{Qian2024error}.
Suppose that $\cC = \{\cC_i\}_{i = 1}^{p}$ is the clustering result produced by a tested algorithm with $\cC_i$ being the set of data points assigned to the $i$-th cluster.
The ground-truth clustering is represented by $\cC\uast = \{\cC_i\uast\}_{i = 1}^{p}$ with each $\cC_i\uast$ defined in the same manner.
Let $n_i = \#|\cC_i|$, $n_j\uast = \#|\cC_j\uast|$, and $n_{i, j} = \#|\cC_i \cap \cC_j\uast|$, where $\#|\,\cdot\,|$ denotes the cardinality of a set.
Then entropy, purity, and NMI are computed as
\begin{equation*}
	\left\{
	\begin{aligned}
		& \mathrm{Entropy} = - \dfrac{1}{n \log_2 p} \sum_{i = 1}^{p} \sum_{j = 1}^{p} n_{i, j} \log_2 \dfrac{n_{i, j}}{n_j\uast}, \quad
		\mathrm{Purity} = \dfrac{1}{n} \sum_{j = 1}^{p} \max_{i = 1, \dotsc, p} \hkh{ n_{i, j} }, \\
		& \mathrm{NMI} = \dfrac{1}{\max \{h (\cC), h (\cC\uast)\}} \sum_{i = 1}^{p} \sum_{j = 1}^{p} \dfrac{n_{i, j}}{n} \log_2 \dfrac{n n_{i, j}}{n_i n_j\uast},
	\end{aligned}
	\right.
\end{equation*}
respectively.
Here, $h (\cC) = - \sum_{i = 1}^{p} (n_i / n) \log_2 (n_i / n)$ and $h (\cC\uast)$ is defined analogously.
Broadly speaking, a clustering assignment is considered more favorable when it yields lower value of entropy and higher values of purity and NMI.

Figure \ref{fig:onmf} illustrates the clustering results of the tested algorithms on image and text datasets, including CPU time and three criteria. 
We use the results of \SUPPORT as the reference baseline in the last three subplots, and accordingly report all quantities after subtracting the corresponding values produced by \SUPPORT.
As shown, all algorithms deliver comparable results in terms of entropy, purity, and NMI, which suggests that they achieve similar clustering qualities.
Nevertheless, the proposed algorithm \SUPPORT stands out in computational efficiency, reducing the CPU time by more than an order of magnitude.
These numerical findings provide evidence that, the superior performance of \SUPPORT is not confined to simulated cases, but also extends to practical applications.

\begin{figure}[t]
	\centering
	\subfigure[CPU Time in Seconds]{
		\label{subfig:onmf_time}
		\includegraphics[width=0.45\linewidth]{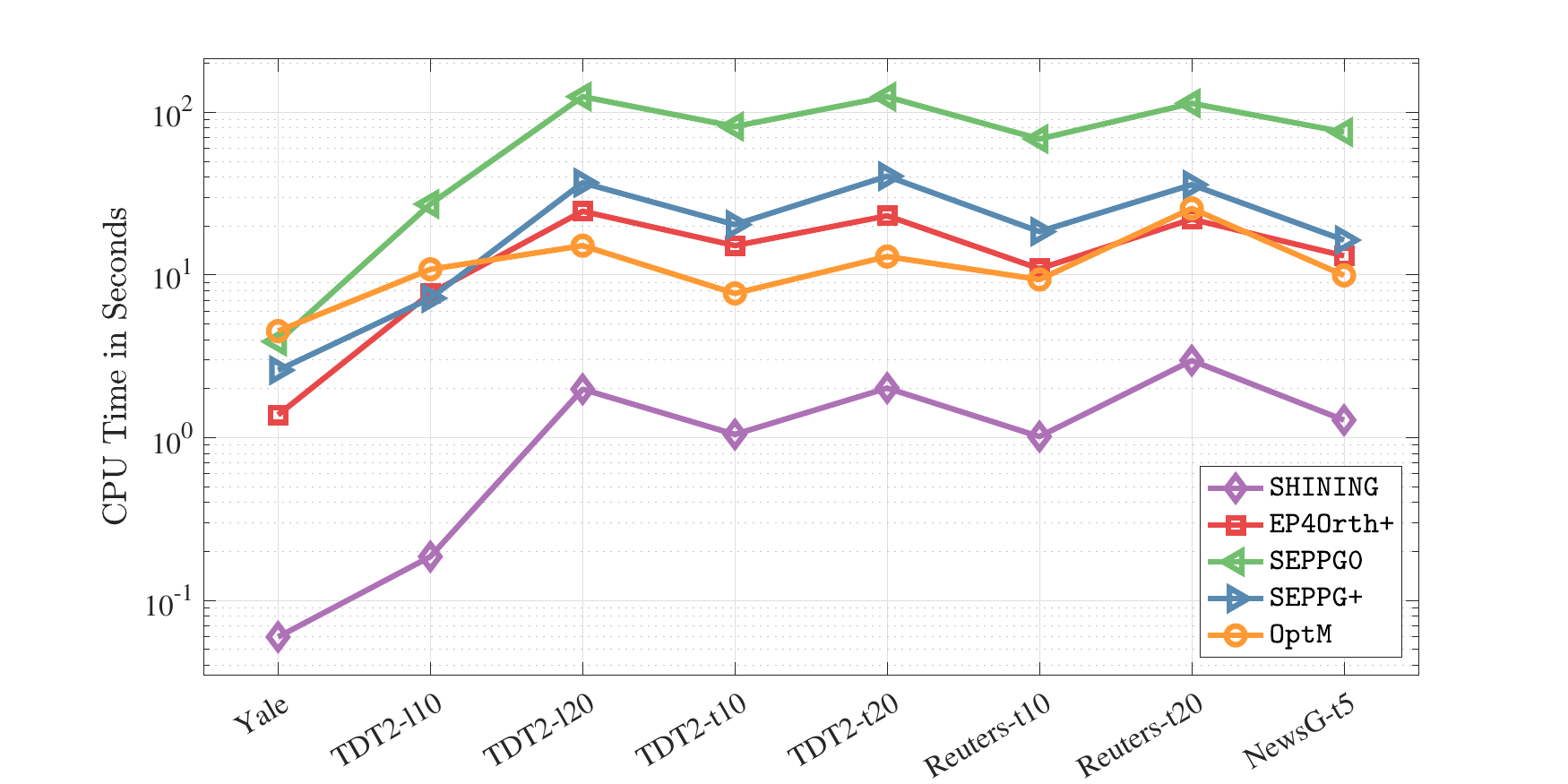}
	}
	\subfigure[Entropy]{
		\label{subfig:onmf_entr}
		\includegraphics[width=0.45\linewidth]{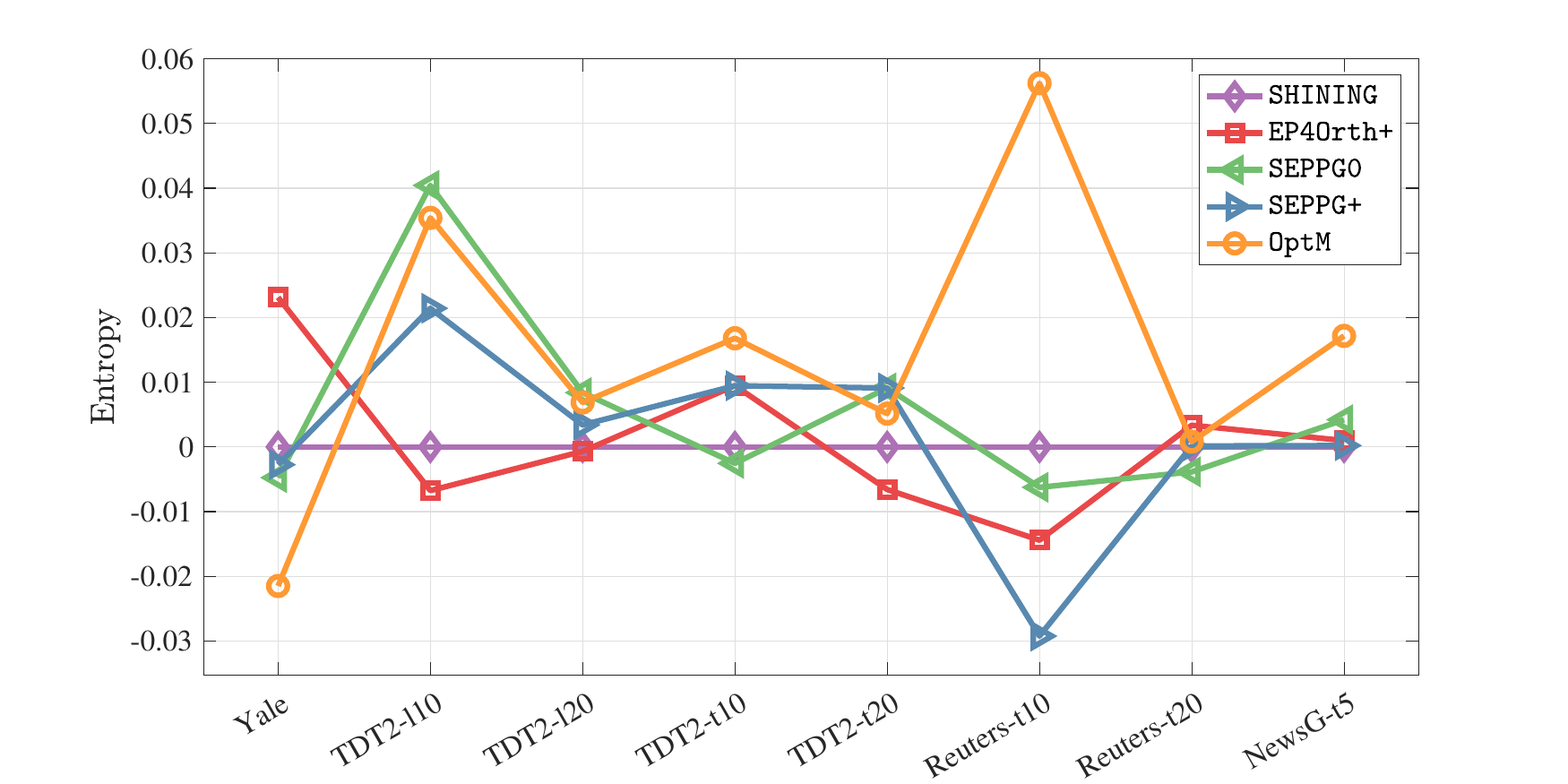}
	}
	
	\subfigure[Purity]{
		\label{subfig:onmf_puri}
		\includegraphics[width=0.45\linewidth]{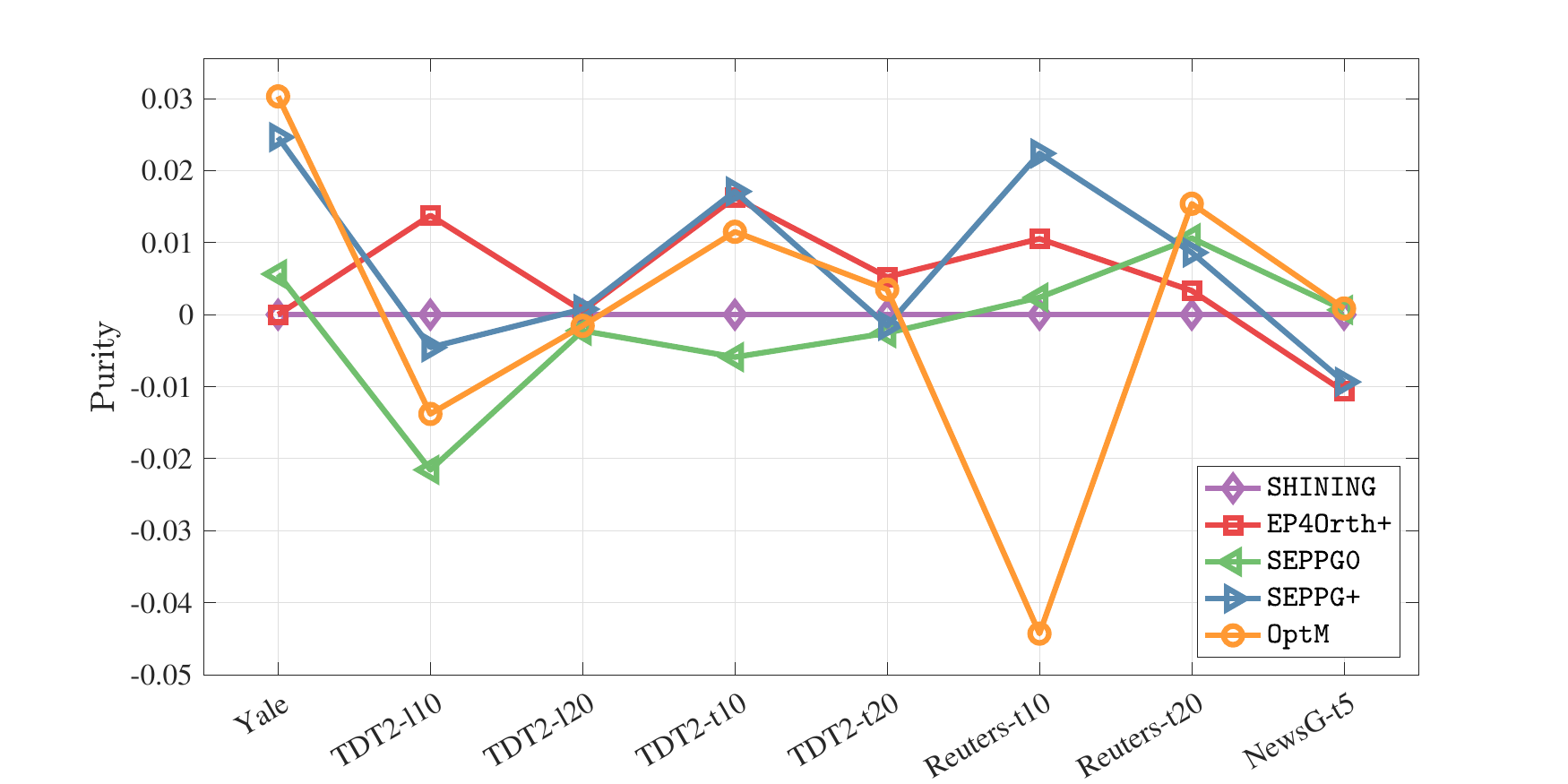}
	}
	\subfigure[NMI]{
		\label{subfig:onmf_nmif}
		\includegraphics[width=0.45\linewidth]{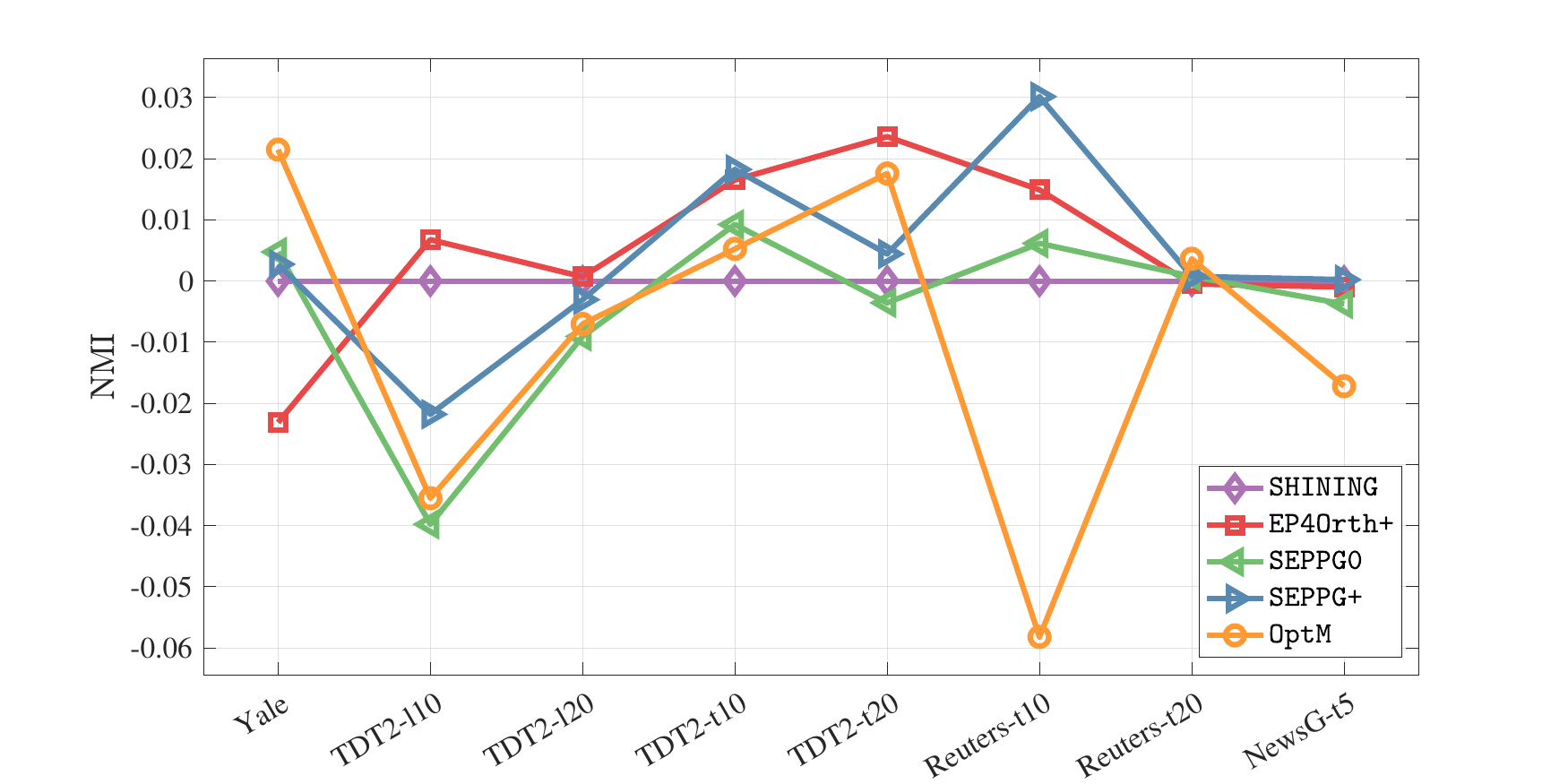}
	}
	\caption{Numerical comparison of different algorithms for clustering.}
	\label{fig:onmf}
\end{figure}

\subsection{Community Detection}

In the final stage of numerical experiments, we investigate the performance of various algorithms on the problem of community detection, a fundamental task in network science with far-reaching implications for machine learning and data analysis.
The goal is to divide a given network, consisting of $n$ vertices and $m$ edges, into $p$ groups in which the connections within each group are markedly denser than those across groups.
Recently, Paul and Chen \cite{Paul2025orthogonal} propose to tackle this problem by solving the optimization model below,
\begin{equation}
	\label{opt:cd}
	\min_{X \in \Opnp} \hspace{2mm} - \dfrac{1}{4} \norm{X\zz A X}\ffs,
\end{equation}
where $A \in \Rnn$, constructed from the adjacency matrix and the normalized Laplacian, encapsulates the structural information of the underlying network.

We select six real-world datasets from GitHub\footnote{See \url{https://github.com/PanShi2016/Community_Detection}.} for our experiments, with a detailed description of each dataset provided in Table \ref{tb:cd}.
The initial points for the algorithms under test are generated based on the eigenvectors of $A$ corresponding to the largest $p$ eigenvalues.
Let $d_i$ and $d_i\uast$ be the predicted group by an algorithm and the ground-truth group of the $i$-th vertex, respectively.
To evaluate the quality of detection outcomes, we employ the accuracy \cite{Xu2003document} as a quantitative metric as follows,
\begin{equation*}
	\mathrm{Accuracy} = \dfrac{1}{n} \sum_{i = 1}^{n} \chi (d_i\uast, \map (d_i)),
\end{equation*}
where $\chi (a, b)$ denotes the indicator function, taking the value $1$ if $a = b$ and $0$ otherwise, and $\map (\cdot)$ represents the permutation mapping function \cite{Xu2003document}.
Clearly, a higher value of accuracy signifies a better detection performance.

\begin{table}[t]
	\centering
	\begin{minipage}{\textwidth}
		\caption{Description of datasets for community detection.}
		\label{tb:cd}
		\begin{tabular*}{\textwidth}{@{\hspace{5pt}\extracolsep{\fill}}ccccccc@{\hspace{5pt}}}
			\toprule
			Dataset & TerrorAttack & CiteSeer & Cora & Email & PubMed & BlogCatalog \\
			\midrule
			$n$ & $1293$ & $3312$ & $2708$ & $1005$ & $19717$ & $10312$ \\
			$m$ & $6344$ & $9072$ & $10556$ & $32128$ & $88648$ & $667966$ \\
			$p$ & $6$ & $6$ & $7$ & $42$ & $3$ & $39$ \\
			\bottomrule
		\end{tabular*}
	\end{minipage}
\end{table}

The numerical results of our testing, presented in Figure \ref{fig:cd}, report both CPU time and accuracy. 
To enhance visual clarity, the values of accuracy are displayed relative to \SUPPORT by subtracting its corresponding values throughout. 
It can be observed that the detection results of the tested algorithms attain roughly comparable accuracy. 
Furthermore, \SUPPORT continues to exhibit a substantial computational advantage, requiring less than one-tenth of the CPU time compared with existing methods. 
This remarkable improvement highlights the practical effectiveness of \SUPPORT in dealing with complex datasets.

\begin{figure}[t]
	\centering
	\subfigure[CPU Time in Seconds]{
		\label{subfig:cd_time}
		\includegraphics[width=0.45\linewidth]{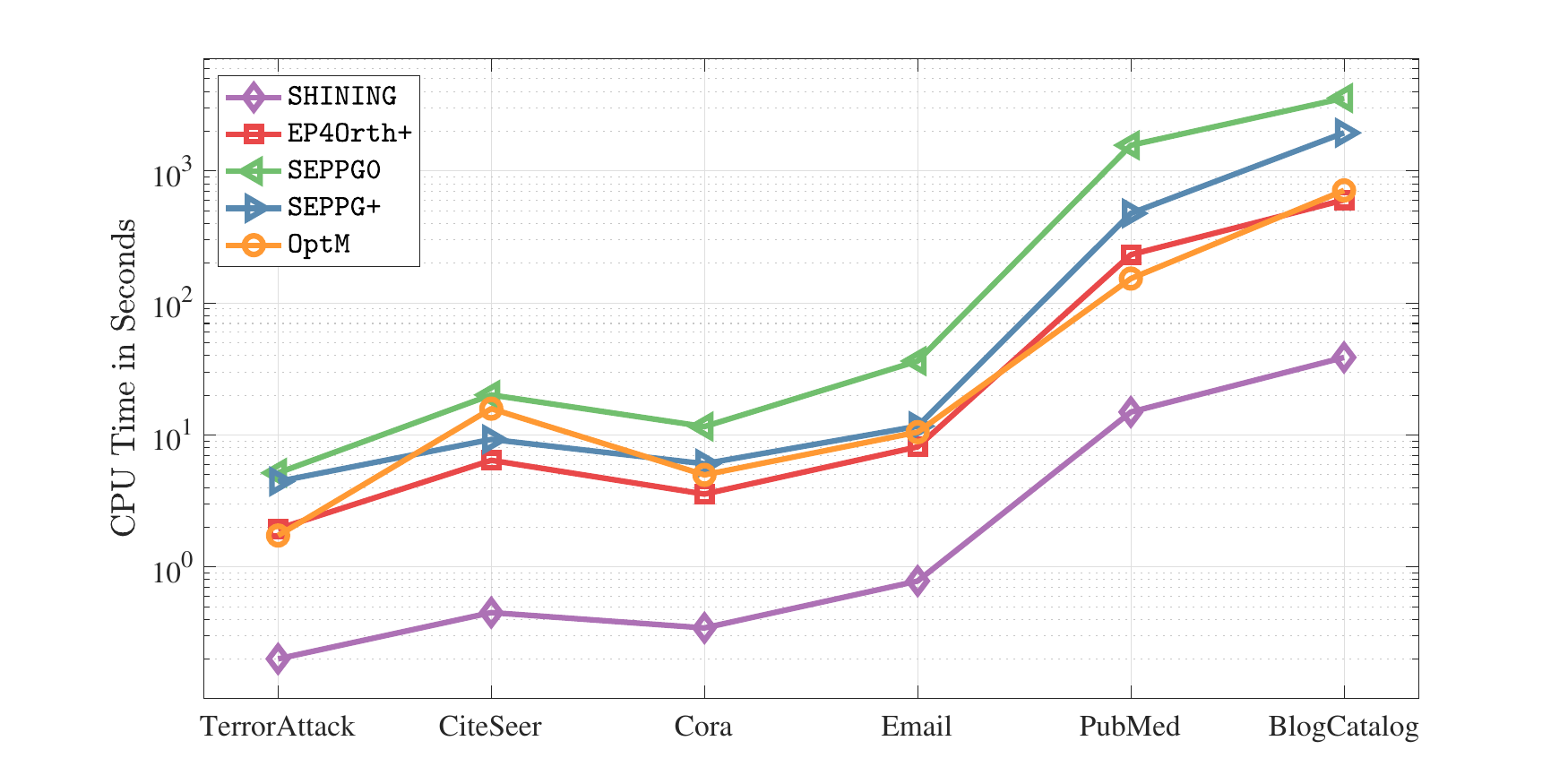}
	}
	\subfigure[Accuracy]{
		\label{subfig:cd_acc}
		\includegraphics[width=0.45\linewidth]{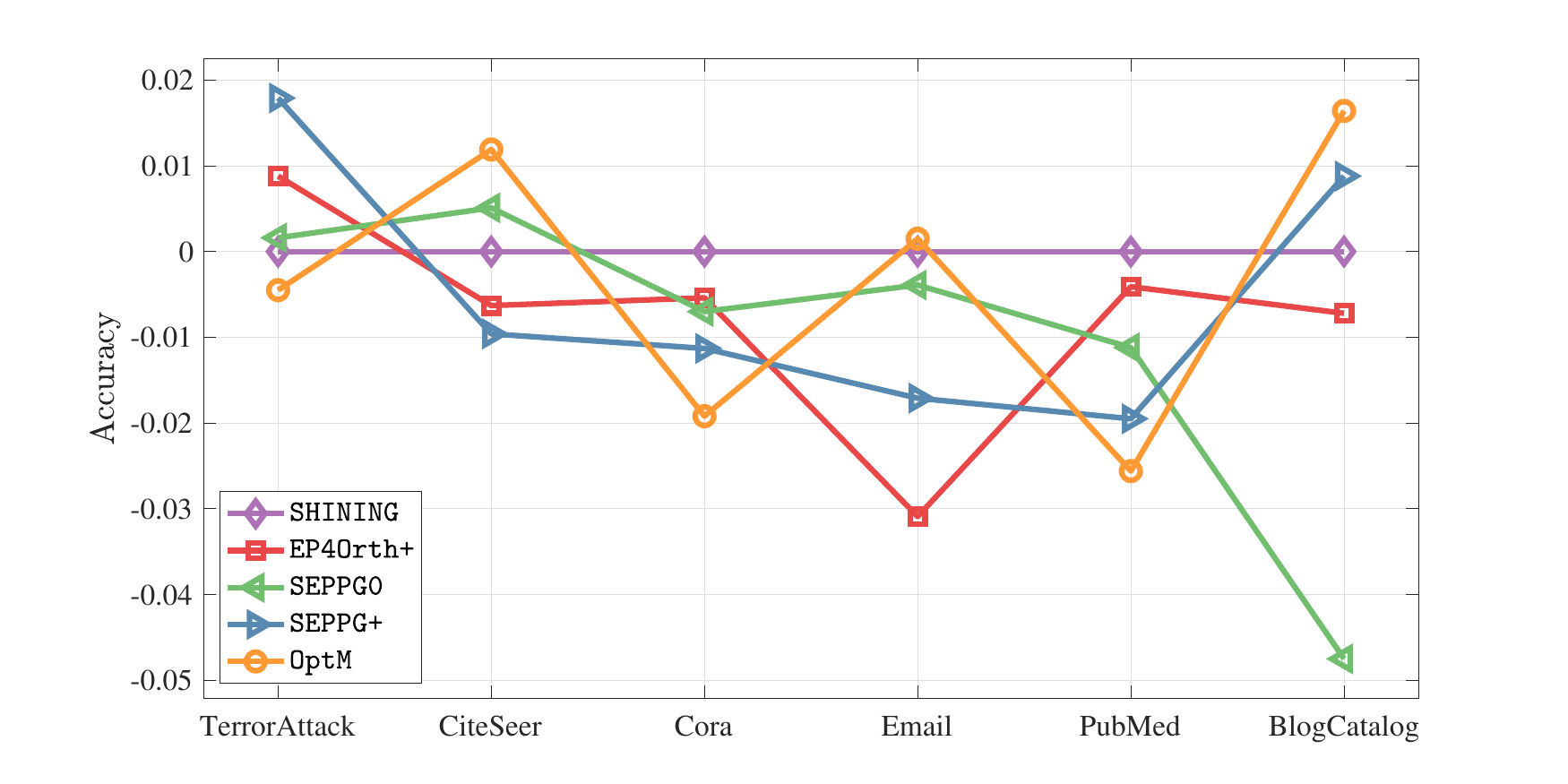}
	}
	\caption{Numerical comparison of different algorithms for community detection.}
	\label{fig:cd}
\end{figure}

\section{Concluding Remarks}

\label{sec:conclusion}

In this paper, we propose a principled and feasible algorithm \SUPPORT for problem \eqref{opt:stplus} by leveraging the property that each row of a matrix in $\Opnp$ contains at most one nonzero entry. 
Our algorithm systematically updates the positions of nonzero entries by monitoring the objective function value and the support set of the current iterate. 
%the support set by adjusting
For rows whose norms are close to zero, the nonzero entry is relocated to a column that results in a further decrease of the objective function value. 
For zero rows, a specific position is directly activated to introduce a nonzero entry. 
Once the support set is updated, the proximal linearization of the objective function is minimized within it until no sufficient reduction is observed, where the corresponding subproblem admits a closed-form solution. 
We establish the subsequence convergence and iteration complexity of \SUPPORT to a first-order stationary point. 
The convergence of the whole sequence can be guaranteed when the objective function is semi-algebraic. 
In addition, our algorithm is capable of identifying the support of stationary points in a finite number of iterations. 
Numerical experiments demonstrate that \SUPPORT has a strong potential to deliver a cutting-edge performance in real-world applications. 
%We believe that the insights and algorithmic principles presented herein will inspire further advances in structured nonconvex optimization problems. 
For future studies, we are interested in developing second-order algorithms to solve optimization problems with nonnegative and orthogonal constraints.

\section*{Acknowledgments}

We would like to express our gratitude to Dr. Yitian Qian and Dr. Lianghai Xiao for kindly sharing the codes of \SEPPGz and \SEPPGp.

% ---------------------------------------------------------------------------------------------------------------------------------

\bibliographystyle{abbrv}

\bibliography{Library_SuppS}

\addcontentsline{toc}{section}{References}

\end{document}